\documentclass[12pt]{amsart}
\usepackage{amssymb}
\usepackage{enumerate}
\usepackage{eucal}
\usepackage{amsmath}
\usepackage{amscd}
\usepackage{txfonts}      
\usepackage[dvips]{color}
\usepackage{mathtools}
\usepackage{multicol}
\usepackage[all]{xy}           
\usepackage{graphicx}
\usepackage{color}
\usepackage{colordvi}
\usepackage{xspace}
\usepackage[normalem]{ulem}
\usepackage{cancel}
\usepackage{tikz}
\usepackage{longtable}
\usepackage{multirow}
\usepackage{ifpdf}
\ifpdf
\usepackage[colorlinks,final,backref=page,hyperindex]{hyperref}
\else
\usepackage[colorlinks,final,backref=page,hyperindex,hypertex]{hyperref}
\fi

\newtheorem{theorem}{Theorem}[section]
\newtheorem{prop}[theorem]{Proposition}
\newtheorem{lemma}[theorem]{Lemma}
\newtheorem{coro}[theorem]{Corollary}
\newtheorem{thm-def}[theorem]{Theorem-Definition}
\newtheorem{def-prop}[theorem]{Definition-Proposition}
\newtheorem{prop-def}[theorem]{Proposition-Definition}
\newtheorem{coro-def}[theorem]{Corollary-Definition}

\newtheorem{question}[theorem]{Question}
\newtheorem{conjecture}[theorem]{Conjecture}
\theoremstyle{definition}
\newtheorem{defn}[theorem]{Definition}
\newtheorem{remark}[theorem]{Remark}

\newtheorem{exam}[theorem]{Example}

\newcommand{\nc}{\newcommand}
\nc{\tred}[1]{\textcolor{red}{#1}}
\nc{\tblue}[1]{\textcolor{blue}{#1}}
\nc{\tgreen}[1]{\textcolor{green}{#1}}
\nc{\tpurple}[1]{\textcolor{purple}{#1}}
\nc{\btred}[1]{\textcolor{red}{\bf #1}}
\nc{\btblue}[1]{\textcolor{blue}{\bf #1}}
\nc{\btgreen}[1]{\textcolor{green}{\bf #1}}
\nc{\btpurple}[1]{\textcolor{purple}{\bf #1}}

\renewcommand{\frak}{\mathfrak}
\newcommand{\efootnote}[1]{}
\renewcommand{\textbf}[1]{}
\newcommand{\delete}[1]{{}}

\delete{
	\nc{\mlabel}[1]{\label{#1} {{\tiny\tt (#1)}}\ }
	\nc{\mcite}[1]{\cite{#1} {{\tiny\tt (#1)}}\ }
	\nc{\mref}[1]{\ref{#1}{{\tiny\tt (#1)}}\ }
	\nc{\meqref}[1]{~\eqref{#1}{{\tiny\tt (#1)}}\ }
	\nc{\mbibitem}[1]{\bibitem[\bf #1]{#1}}
}

\nc{\mlabel}[1]{\label{#1}}  
\nc{\mcite}[1]{\cite{#1}}  
\nc{\mref}[1]{\ref{#1}}  
\nc{\meqref}[1]{~\eqref{#1}}
\nc{\mbibitem}[1]{\bibitem{#1}} 

\nc{\sbar}{, }
\delete{
	{\, {\scriptstyle{|\hspace{-.08cm}|\hspace{-.08cm}|
				\hspace{-.08cm}|\hspace{-.08cm}|\hspace{-.08cm}|
				\hspace{-.08cm}|\hspace{-.08cm}|\hspace{-.08cm}|\hspace{-.08cm}|
				\hspace{-.08cm}|\hspace{-.08cm}|\hspace{-.08cm}|
				\hspace{-.08cm}|\hspace{-.08cm}|\hspace{-.08cm}|\hspace{-.08cm}|
				\hspace{-.08cm}|\hspace{-.08cm}|\hspace{-.08cm}|
				\hspace{-.08cm}|\hspace{-.08cm}|\hspace{-.08cm}|\hspace{-.08cm}|
				\hspace{-.08cm}|\hspace{-.08cm}|\hspace{-.08cm}|
				\hspace{-.08cm}|\hspace{-.08cm}|}\, }}
}

\nc{\wvec}[2]{{\scriptsize{ [
			\begin{array}{c} #1 \\ #2 \end{array}   ]}}}
\nc{\lp}{\big ( }
\nc{\llp}{\Big (}
\nc{\Llp}{\left (}
\nc{\rp}{\big ) }
\nc{\rrp}{\Big )}
\nc{\Rrp}{\right )}
\nc{\lb}{\big < }
\nc{\llb}{\!\Big \langle }
\nc{\Llb}{\! \left <}
\nc{\rb}{\big >  }
\nc{\rrb}{\Big \rangle \!}
\nc{\Rb}{\Big \rangle\! }
\nc{\length}{{\rm leng}}

\nc{\bin}[2]{ (_{\stackrel{\scs{#1}}{\scs{#2}}})}  
\nc{\binc}[2]{ \big (\! \begin{array}{c} \scs{#1}\\
		\scs{#2} \end{array}\! \big )}  
\nc{\bincc}[2]{  \left ( {\scs{#1} \atop
		\vspace{-1cm}\scs{#2}} \right )}  
\nc{\bs}{\bar{S}}
\nc{\cosum}{\sqsubset}
\nc{\la}{\longrightarrow}
\nc{\rar}{\rightarrow}
\nc{\dar}{\downarrow}
\nc{\dap}[1]{\downarrow \rlap{$\scriptstyle{#1}$}}
\nc{\uap}[1]{\uparrow \rlap{$\scriptstyle{#1}$}}
\nc{\defeq}{\stackrel{\rm def}{=}}
\nc{\disp}[1]{\displaystyle{#1}}
\nc{\dotcup}{\ \displaystyle{\bigcup^\bullet}\ }
\nc{\gzeta}{\bar{\zeta}}
\nc{\hcm}{\ \hat{,}\ }
\nc{\hts}{\hat{\otimes}}
\nc{\barot}{{\otimes}}
\nc{\free}[1]{\bar{#1}}
\nc{\uni}[1]{\tilde{#1}}          
\nc{\hcirc}{\hat{\circ}}
\nc{\lleft}{[}
\nc{\lright}{]}
\nc{\curlyl}{\left \{ \begin{array}{c} {} \\ {} \end{array}
	\right .  \!\!\!\!\!\!\!}
\nc{\curlyr}{ \!\!\!\!\!\!\!
	\left . \begin{array}{c} {} \\ {} \end{array}
	\right \} }
\nc{\longmid}{\left | \begin{array}{c} {} \\ {} \end{array}
	\right . \!\!\!\!\!\!\!}
\nc{\ora}[1]{\stackrel{#1}{\rar}}
\nc{\ola}[1]{\stackrel{#1}{\la}}
\nc{\ot}{\otimes}
\nc{\mot}{{{\sbar}}}
\nc{\otm}{\mot}
\nc{\scs}[1]{\scriptstyle{#1}}
\nc{\subv}{{^{\star}}}
\nc{\cov}{{^{\sharp}}}
\nc{\mrm}[1]{{\rm #1}}
\nc{\dirlim}{\displaystyle{\lim_{\longrightarrow}}\,}
\nc{\invlim}{\displaystyle{\lim_{\longleftarrow}}\,}
\nc{\proofbegin}{\noindent{\bf Proof: }}
	\nc{\proofend}{$\quad \square$ \vspace{0.3cm}}
\nc{\sha}{{\mbox{\cyr X}}}  
\nc{\shap}{{\mbox{\cyrs X}}} 
\nc{\shpr}{\diamond}    
\nc{\shplus}{\shpr^+}
\nc{\shprc}{\shpr_c}    
\nc{\msh}{\ast}
\nc{\vep}{\varepsilon}
\nc{\labs}{\mid\!}
\nc{\rabs}{\!\mid}

\newcommand{\C}{\mathbb{C}}

\newcommand{\Q}{\mathbb{Q}}
\newcommand{\R}{\mathbb{R}}

\newcommand{\Z}{\mathbb{Z}}

\newcommand {\cala}{{\mathcal {A}}}

\newcommand {\calh}{{\mathcal {H}}}

\newcommand {\calm}{{\mathcal {M}}}
\newcommand {\calp}{{\mathcal {P}}}

\nc{\fraka}{{\frak a}}
\nc{\frakA}{{\frak A}}
\nc{\frakb}{{\frak b}}
\nc{\frakB}{{\frak B}}
\nc{\frakf}{{\frak F}}
\nc{\frakh}{{\frak h}}
\nc{\frakH}{{\frak H}}
\nc{\frakk}{{\frak k}}
\nc{\frakK}{{\frak K}}
\nc{\frakM}{{\frak M}}
\nc{\frakm}{{\frak m}}
\nc{\frakP}{{\frak P}}
\nc{\frakp}{{\frak p}}
\nc{\frakS}{{\frak S}}

\nc{\bfrakM}{\overline{\frakM}}

\nc {\e} {{\epsilon}}
\nc{\fpower}{\calp_{\rm fin}}
\nc{\pfpair}[2]{\Big(\begin{array}{c}\scs{#1} \\ \scs{#2} \end{array} \Big)}

\font\cyr=wncyr10
\font\cyrs=wncyr7
\nc{\zb}[1]{\textcolor{blue}{Bin: #1}}
\nc{\xhy}[1]{\textcolor{red}{#1}}
\nc{\li}[1]{\textcolor{purple}{#1}}
\nc{\lir}[1]{\textcolor{purple}{Li: #1}}

\newcommand {\calha}{{{\mathcal {H}} _{\Z }}}
\newcommand {\calhb}{{{\mathcal {H}} _{\Z _{\ge 0}}}}
\newcommand {\calhc}{{{\mathcal {H}} _{\Z _{\le 0}}}}
\newcommand {\calhd}{{{\mathcal {H}} _{\Z _{\ge 1}}}}

\newcommand {\calia}{{{\mathcal {H}} _{\Z }^+}}
\newcommand {\calib}{{{\mathcal {H}} _{\Z _{\ge 0}}^+}}

\newcommand {\calid}{{{\mathcal {H}} _{\Z _{\ge 1}}}}

\newcommand {\rmmin} {{\rm min}}
\newcommand {\fch} {{F^{\mathrm{Ch}}}}
\newcommand {\sch} {{S^{\mathrm{Ch}}}}

\nc {\cks}{\text{\textcircled {s}}\xspace}

\begin{document}
	
	\title[Extended shuffle product]{Extended Shuffle product for multiple zeta values }
	%
	
\author{Li Guo}
\address{Department of Mathematics and Computer Science,
Rutgers University, Newark, NJ 07102, USA}
\email{liguo@rutgers.edu}	

\author{Wenchuan Hu}
\address{School of Mathematics,
Sichuan University, Chengdu, 610064, P. R. China}
\email{wenchuan@scu.edu.cn}

\author{Hongyu Xiang}
\address{School of Mathematics,
Sichuan University, Chengdu, 610064, P. R. China}
\email{xianghongyu1@stu.scu.edu.cn}

\author{Bin Zhang}
\address{School of Mathematics,
Sichuan University, Chengdu, 610064, P. R. China}
\email{zhangbin@scu.edu.cn}
	
	\date{\today}


\begin{abstract} The shuffle algebra on positive integers encodes the usual multiple zeta values (MZVs) (with positive arguments) thanks to the representations of MZVs by iterated Chen integrals of Kontsevich. Together with the quasi-shuffle (stuffle) algebra, it provides the algebraic framework to study relations among MZVs.
	
This paper enlarges the shuffle algebra uniquely to what we call the extended shuffle algebra that encodes convergent multiple zeta series with arbitrary integer arguments, not just the positive ones in the usual case. To achieved this goal, we first replace the Rota-Baxter operator of weight zero (the integral operator) that characterizes the shuffle product by the differential operator which extends the shuffle product to the larger space. We then show that the subspace corresponding to the convergent MZVs with integer arguments becomes a subalgebra under this extended shuffle product. Furthermore, by lifting the extended shuffle algebra to the locality algebra of Chen symbols, we prove that taking summations of fractions from Chen symbols defines an algebra homomorphism from the above subalgebra to the subalgebra of real numbers spanned by convergent multiple zeta series.
\end{abstract}

\subjclass[2020]{
11M32,	
16W25,	
17B38,	
16S10,	
40B05	
}
	
\keywords{multiple zeta value, shuffle product, Chen symbol, derivation, differential operator, Rota-Baxter operator, locality algebra}

	\maketitle

	\tableofcontents
	
	\setcounter{section}{0}
	
	\section{Introduction}
	Recall that \mcite{AET,Zj} the multiple zeta series
	\begin {equation}
	\mlabel {eq:zeta}
\zeta(\vec{s}):=	\sum _{n_1>\cdots > n_k>0}\frac 1{n_1^{s_1}\cdots n_k^{s_k}}
\end{equation}
converges in the region
\begin{equation}\mlabel{eq:conv}
\{\vec s \in \C ^k \ |\ \Re (s_1)+\cdots + \Re (s_i)>i, \ i=1, \cdots , k\},
\end{equation}
thus defining a holomorphic function in this region. This holomorphic function can be analytically continued to a meromorphic function $\zeta (\vec s)$ in $\C ^k$, called {\bf multiple zeta function} on $\C^k$, with simple poles at
\begin{equation*}
	\{s_1=1\}\bigcup  \{ s_{1}+s_2=2,1,0,-2,-4, \cdots \} \bigcup \bigcup_{ j=3}^k \Big\{\sum_{i=1}^j s_{i} \in \Z_{\leq j}\Big\}.
\end{equation*}
When the value of a multiple zeta function at a positive integer point is convergent, it is called a {\bf multiple zeta value} (MZV).  This is the case for $\zeta (\vec s )$ with $s_1\in \Z_{>1}$, $s_i\in \Z_{\ge 1}, 2\le i\le k$.

\subsection {Motivation}
There are rich linear relations among the MZVs over  $\Q$. Determining the structure of the relations is one of the main goals in the study of MZVs.
There are two fundamental families of relations:
\begin{enumerate}
	\item the stuffle (or quasi-shuffle) relation, based on the series expression of MZVs:
	\begin {equation*}
	\zeta (s_1, \cdots, s_k)=\sum _{n_1>\cdots > n_k>0}\frac 1{n_1^{s_1}\cdots n_k^{s_k}}.
\end{equation*}
\item the shuffle relation, based on the iterated integral expression of MZVs~\mcite{Zag}:
\begin {equation}
	\mlabel {eq:zetai}
	\zeta (s_1, \cdots s_k)=\underbrace {\int_ 0^1\frac {dt}t\int _0^t \frac {dt}t\cdots \int _0^t\frac {dt}t}_{s_1-1}\int _0^t\frac {dt}{1-t}\cdots \underbrace {\int _0^t\frac {dt}t\cdots \int _0^t\frac {dt}t}_{s_k-1}\int _0^t\frac {dt}{1-t}.
	\end{equation}
 Here the subscripts of the variable $t_i, 1\leq i\leq s_1+\cdots+s_k,$ are suppressed for simplicity.
\end{enumerate}

These two families of relations can be encoded into two associative multiplications $\ast $ and $\shap$ in the vector space
$$\calh^{0}:=\Q \oplus \bigoplus _{k\in \Z _{>0}, \vec s\in \Z _{\ge 1} ^k, s_1\ge 2 }\Q [\vec s],$$
which is the space spanned by integer points $\vec{s}$ corresponding to MZVs and the unit $1$, in the sense that the maps
$$\zeta^\ast:(\calh^{0}, \ast )\longrightarrow \R , \quad \zeta^\shap: (\calh^{0}, \shap)\longrightarrow \R, $$
both sending $\vec{s}$ to $\zeta(\vec{s})$, are algebra homomorphisms, that is,
$$\zeta^\ast([\vec s]\ast [\vec t])=\zeta ([\vec s])\zeta ([\vec t])=\zeta^\shap([\vec s]\shap [\vec t]).
$$
Consequently,
$$\Big\{\zeta^\ast ([\vec s] \ast [\vec t])-\zeta^\shap ([\vec s] \shap [\vec t])\ \Big| \ \vec s\in \Z_{\ge 1}^k, s_1>1, \vec t\in \Z_{\ge 1}^\ell , t_1>1, k, \ell \in \Z_{>0} \Big\}
$$
is a family of $\Q$-linear relations among MZVs, called the {\bf double shuffle relations}.

The double shuffle relations do not exhaust all known relations among MZVs.  For example, $\zeta (3)=\zeta (2,1)$ is not in this family. To suitably enlarge the framework, Ihara-Kaneko-Zagier \mcite{IKZ} extended the algebra homomorphisms $\zeta^\ast$ and $\zeta^\shap$ to the space
\begin{equation}
\calh _{\Z _{\ge 1}}:=\Q\oplus \bigoplus _{k\in \Z _{>0}, \vec s\in \Z _{\ge 1} ^k }\Q [\vec s],
\mlabel{eq:h0}
\end{equation}
which also carries the stuffle multiplication $\ast$ and shuffle multiplication $\shap$,  yielding algebra homomorphisms
\begin{equation} \mlabel{eq:exthom}
\zeta^{\ast}: (\calhd, \ast)\longrightarrow \R[T], \quad  \zeta^{\shap}: (\calhd, \shap)\longrightarrow \R[T].
\end{equation}
They call the family of relations
$$\Big\{[1]\ast [\vec w]-[1]\shap [\vec w], [\vec s] \ast [\vec t]-[\vec s] \shap [\vec t]\ \Big| \ [\vec w], [\vec s], [\vec t]\in  \calh^0 \Big\}
$$
the {\bf extended double shuffle relations}.

\begin{conjecture}\mcite{IKZ} The extended double shuffle relation gives all $\Q$-linear relations among MZVs.
\end{conjecture}

Despite great efforts, the conjecture is still out of reach by the current approaches except remarkable progresses on the motivic level such as~\mcite{Br,Za2}. Further insight to the conjecture might be obtained by revealing hidden structures of multiple zeta values to provide further insight. Here a hidden structure can be in the positive integer region, or in the convergent region, or even in the divergent region, of the MZVs.
So far most investigations have been carried out in the positive integer region. But the whole region where the multiple zeta series are convergent should also be a good candidate to explore. This is the motivation of this paper.

\subsection {Our approach} In exploring the algebraic structure underlying multiple zeta values with arbitrary integer arguments, the following question naturally arises.

\begin{question}
Can the stuffe and shuffle products in $\calh _{\Z _{\ge 1}}$ be further extended to the space
$$\calh ^0_\Z :=\Q{\bf 1}\oplus\bigoplus_{\vec s\in \Z^k, s_1+\cdots +s_j>j, j=1, \cdots, k, k\ge 1} \Q[\vec s],$$
that corresponds to the convergent multiple zeta series with arbitrary integer arguments?
\end{question}

Note that $\calh^0_\Z$ contains $\calh^0$ in Eq.~\meqref{eq:h0} that is usually considered for MZVs (with positive arguments). Indeed $\calh^0_\Z\cap \Z_{>0}^{(\infty)} =\calh^0$ for $\Z_{>0}^{(\infty)}:=\cup_{k\geq 1} \Z_{>0}^k$. So the notion $\calh^0_\Z$ is consistent with $\calh^0$.

It is obvious that stuffle product can be extended to the algebra of all integer points:
$$\calh _{\Z }:=\Q{\bf1} \oplus\bigoplus _{k\in \Z _{>0}, \vec s\in \Z  ^k }\Q [\vec s],
$$
and convergent points are closed under this extended stuffle product.

The situation is quite different to extend the shuffle product to $\calh_{\Z}$, for the lack of an iterated integral expression like Eq.~\meqref{eq:zetai} for convergent integer points with possibly nonpositive entries. So we need a devise to bypass the integral expression for MZVs to extend the shuffle product.

In \mcite {GZ1}, the shuffle product of MZVs is characterized as the unique product which makes a generalization of the integral operator (the operator $I$ in Eq. (\mref {eq:Idefn})) a Rota-Baxter operator, thus giving a method to construct the shuffle product without the integral expressions.

We follow this idea to extend the shuffle product to $\calh _{\Z }$. Though the operator $I$ can be extended from positive integer points to all integer points easily, it is hard to define its image for the unit. So we work with its ``inverse" operator $J$ by requiring the operator $J$ to be a differential operator. This requirement together with some other initial conditions uniquely determines the extended shuffle product.

Once the extended shuffle product is defined, we need to justify that this is the structure underlying convergent multiple zeta series by checking that:
\begin{enumerate}
  \item the subspace $\calh_{\Z}^0$ of $\calh_{\Z}$ corresponding to convergent multiple zeta series is a subalgebra,
\item evaluations for multiple zeta series defines an algebra homomorphism from this subalgebra to $\R$.
\end{enumerate}

For the first task, since the criterion for convergency is about partial weights as shown in Eq.~\meqref{eq:conv}, it is natural to ask for an estimation of partial weights of terms in the product. The possible negative entries make the estimation complicated, and we have to modify the concept of partial weights. This gives an estimation of partial weights for terms in the product in terms of modified partial weights.

The second task is quite involved. Parallel to the treatment of MZVs in \mcite {GX}, we generalize the concept of Chen fractions, and introduce its abstract version, called Chen symbols, to model the extended shuffle product. Then we use the technical tool of locality algebras \mcite {CGPZ,CGPZ2,GPZ3}, to reveal the structure of the space of Chen symbols and the connection between Chen symbols and Chen fractions.

With all these preparations, we can show that taking the multiple zeta series gives an algebra homomorphism from the algebra of convergent points to the algebra of real numbers, extending the algebra homomorphism in Eq.~\meqref{eq:exthom}.

\subsection {Outline of the paper} The goal of this paper is to explore the algebraic structure for convergent integer points of multiple zeta series, which recovers the shuffle product of MZVs when restricted to positive integer points.

Section \mref{s:shuf} gives the construction of the extended shuffle product (Theorem~\mref{lem:Inverse}).
Based on the idea in \mcite {GZ1}, we give the  construction by using the differential operator $J$, which is roughly the inverse of the Rota-Baxter operator used in \mcite {GZ1}. Essential in this construction is to determine the conditions for the extended shuffle product. Once we have the suitable conditions, the construction of this extended shuffle product follows immediately, and the associativity and uniqueness of the extended shuffle product follow by standard though tedious inductions.

We then need to justify that the constructed extended shuffle product is indeed the structure underneath the MZVs for the convergent integer points. This is accomplished in Section \mref{s:loc} and \mref{s:conv}.

More specifically, in Section \mref{s:loc}, as a model of the extended shuffle product for all integer points and mimicking the summation process, we introduce the concepts of generalized Chen fractions and Chen symbols. Then the extended shuffle product is lifted to a locality product in the space of Chen symbols and a locality algebra homomorphism is provided from the space of Chen symbols to the space of generalized Chen fractions (Proposition~\mref{prop:zhomo}). Working in locality setting is an important step to explore the hidden structures. So we begin the section with recalling some basics for locality algebras.

Then in Section \mref{s:conv}, we first give an estimation of partial weights of the terms appeared in the extended shuffle product in terms of modified partial weights. Based on this estimation, the space spanned by convergent integer points is shown to be a subalgebra of $\calh _\Z$, that is, this space is closed under this extended shuffle product (Theorem~\mref{thm:CAlg}).
After this, it is shown that taking the multiple zeta series actually gives an algebra homomorphism from  the space of convergent integer points to real numbers, bridged by locality algebra of Chen symbols and Chen fractions (Theorem~\mref{thm:zhomo}).

To keep the paper at a reasonable length, further properties of this extended shuffle product is studied in a followup paper \mcite {GXZ}, showing that this extended shuffle product gives a duality of Hopf algebras related to MZVs.

\subsection {Notation}
\mlabel{ss:notn}
We put together notations to be used in the paper.

For a set $X$, denote the set
$$H_X:=\bigcup _{k\in \Z _{\geq 0}}X^k,
$$
and let
$$\calh _X:=\Q H_X=\bigoplus _{k\in \Z _{\geq 0}, \vec s\in X ^k }\Q [\vec s]
$$
denote the vector space with a basis $H_X$. Here the new notion $[\vec{s}]$ is used in place of $\vec{s}$ to emphasize that it is only a symbol, without inheriting any structures of $H_X$ such as the componentwise operations.
Note that $H_X$ share the same underlying set as the free monoid generated by $X$ and $\calh_X$ shares the same underlying space as the noncommutative polynomial algebra $\Q\langle X \rangle$.
Also denote the subset
$$H_X^+:=\bigcup _{k\in \Z _{>0}}X^k
$$
of $H_X$ and let
$$\calh_X^+ := \bigoplus _{k\in \Z _{>0}, \vec s\in X ^k }\Q [\vec s]
$$
be the subspace of $\calh_X$ with basis $H_X^+$.

In this paper, we are mostly interested in taking $X$ to be $\Z$, $\Z_{\ge 0}$, $\Z _{\ge 1}$, $\Z _{\le 0}$, $\Z _{\le 1}$ or $\Z \times \Z _{\ge 1}$ for their relations to MZVs. For $\vec{s}=(s_1,\ldots,s_k)\in X^k$, when we write $\vec{s}=(s_1,\vec{s}\,')$, we correspondingly denote
$[\vec{s}]=[s_1,\vec{s}\,']$.

For a vector $\vec s=(s_1,\cdots, s_k) \in \Z ^k$,

\begin{enumerate}
	\item $k$ is called the {\bf depth} of $[\vec s]$ and denoted by $d([\vec s])$;
	\item the partial sum $w_i([\vec s]):=s_1+\cdots+s_i$ is called the $i$-th {\bf partial weight} of $\vec{s}$, and $w([\vec s]):=w_k([\vec s])$ is called the {\bf weight}. For convenience, let $w_0([\vec s]):=0$, $w({\bf 1}):=w_0({\bf 1}):=0$.
\end{enumerate}

\section {The extended shuffle product on $\calha$}
\mlabel{s:shuf}
The usual shuffle product on $\calh^0$ can be characterized by an integration (a Rota-Baxter operator of weight zero). Applying a derivation instead of an integration, we extend this shuffle product to a product on $\calh_{\Z}$.

\subsection{The shuffle algebra for MZVs} For $\vec s \in \Z ^k_{\ge 1}$, we have iterated integral expression (\mref {eq:zetai}) for MZVs.
This expression gives shuffle relation for MZVs. The integral expression can be put in a simpler form abstractly
\begin {equation}
\mlabel {eq:szeta}
x_0^{s_1-1}x_1\cdots x_0^{s_k-1}x_1.
\end{equation}
Thus an MZV is abstracted to an element in $x_0\Q \langle x_0, x_1 \rangle x_1$, where $\Q \langle x_0, x_1 \rangle$ is the noncommutative polynomial algebra over $\Q$ in variables $x_0, x_1$.

The shuffle relation of MZVs can be put in a more general setup. For any alphabet $\cala$, let $W:=W_\cala$ be the set of words (including the empty word ${\bf 1}$) with letters in $\cala$, and $\Q W$ be the linear space with a basis $W$. Then $\Q W$ carries a shuffle multiplication $\shap$ \mcite{H}, which is inductively defined by
$$w\shap {\bf 1}={\bf 1}\shap w=w,$$
$$a w_1\shap bw_2=a(w_1\shap b w_2)+b(a w_1\shap w_2).
$$
for $w, w_1, w_2 \in \Q W$, $a, b\in \cala$.

Through Eq. \meqref {eq:szeta}, MZVs are represented by words with alphabet $\cala =\{x_0, x_1\}$. In this case, $\Q W$ is the noncommutative polynomial algebra $\Q \langle x_0, x_1 \rangle$, while the subspace $\Q \langle x_0, x_1 \rangle x_1$ (resp. $x_0\Q \langle x_0, x_1 \rangle x_1$), namely the linear space generated by words ending with $x_1$ (resp. words starting with $x_0$ and ending with $x_1$), is a subalgebra under the shuffle multiplication. The subalgebra $\Q {\bf 1}\oplus \Q \langle x_0, x_1 \rangle x_1$ is called  the {\bf shuffle algebra} for MZVs.

With the notions in Section~\mref{ss:notn}, we have a linear isomorphism
\begin {equation}
\mlabel {eq:rho}
\rho: \calid \to \Q \langle x_0, x_1 \rangle x_1, \ [s_1, \cdots, s_k]\to x_0^{s_1-1}x_1\cdots x_0^{s_k-1}x_1,
\end{equation}
the shuffle product $\shap$ in $\Q \langle x_0, x_1 \rangle x_1$ is pulled back to  a product on $\calid$, and extended to $\calhd$ with the unit ${\bf1}$, which we still denote by $\shap$. With this product, $\calhd $ is also called the shuffle algebra for MZVs. Then
$$\calh ^0:=\rho ^{-1}(x_0\Q \langle x_0, x_1 \rangle x_1)=\bigoplus _{k\in \Z _{>0}, \vec s\in \Z _{\ge 1} ^k, s_1\ge 2 }\Q [\vec s]
$$
is a subalgebra which encodes the MZVs in the usual sense, namely, with positive integer arguments.

The shuffle product $\shap$ in $\calhd$ can be extended to $\calhb$ \mcite {GZ1} as follows. Define a linear operator
\begin {equation}
\mlabel {eq:Idefn}
I:\calib \longrightarrow \calib, \quad [s_1, s_2, \cdots, s_k]\mapsto [s_1+1, s_2, \cdots, s_k]
\end{equation}
where $[s_1, s_2, \cdots, s_k]$ is a basis element of $\calib$.
Then the extension of the shuffle product on $\calhb$ is characterized in the following proposition.

 Recall that a {\bf Rota-Baxter operator of weight $0$} on an algebra $A$ is a linear operator $I:A \to A$ such that
$$I(a)I(b)=I(aI(b))+I(I(a)b), \quad a, b\in A.
$$
See \mcite {Guo} for details.

\begin {prop} \mcite {GZ1} There is a unique commutative
multiplication $\shap $ on $\calib$ which makes $I$ into a Rota-Baxter operator of weight $0$ and
$$[0]\shap [\vec s]=[0,\vec s].
$$
It can be extended to $\calhb$ by
$${\bf1}\shap [\vec s]=[\vec s]\shap {\bf1}=[\vec s].
$$
The usual shuffle product on $\calhd$ is the restriction of this product.
\end{prop}

\subsection {Extended shuffle product}
Following the above idea, we can extend the shuffle product further to the space $\calha$. The key ingredient in the construction is the inverse operator of the operator $I$ in Eq. \meqref {eq:Idefn}. The operator $I$ can be extended by the same definition on the basis:
\begin {equation}
\mlabel {eq:NewI}
I:\calia \longrightarrow \calia, \quad [s_1, s_2, \cdots, s_k]\mapsto [s_1+1, s_2, \cdots, s_k].
\end{equation}

Obviously, this operator $I$ is invertible in $\calia$. The inverse is given by
\begin{equation}\mlabel{eq:jmap}
J: \calia \longrightarrow \calia, \quad [s_1, s_2, \cdots, s_k]\mapsto [s_1-1, s_2, \cdots, s_k]
\end{equation}
on a basis element.

A Rota-Baxter operator is closely related to a differential operator as shown in the following well-known and easily checked fact.

\begin {lemma} Let $D: A \to A $ be an invertible linear operator on an algebra $A$. The following statements are equivalent.
\begin{enumerate}
\item $D$ is a differential operator, that is,
$$ D(ab)=D(a)b+aD(b),  \quad a, b\in A.
$$
\item Its inverse $P: A \to A $ is a Rota-Baxter operator of weight $0$.
\end{enumerate}
\end{lemma}

The operator $I$ can not be extended to $\calha$ in a direct way and keep the Rota-Baxter property because of the unit ${\bf 1}$. But $J$ can be extended to $\calha$ easily by taking
$$J({\bf 1})=0.
$$
So it is more convenient to work with the linear operator $J$.

It is worth emphasizing that $J$ and $I$ are mutually inverses on the nonunitary algebra $\calia$. But not on the unitary algebra $\calha$.

The shuffle product on $\calhd$ behaves well with respect to both the weight and depth. Concretely, let
$$W_{\Z_ {\ge 1},0}:=\{{\bf 1}\}, \ D_{\Z_ {\ge 1},0}:=\{{\bf 1}\}.
$$
Also, for $n\in \Z _{\ge 1}$, let
$$W_{\Z_ {\ge 1},n}:=\{[\vec s] \ |\ \vec s \in \Z _{\ge 1}^k, \ w([\vec s])=n\},\ D_{\Z _ {\ge 1},n}:=\{[\vec s] \ | \ \vec s \in \Z _{\ge 1}^n\},
$$
and let $\Q W_{\Z _ {\ge 1},n}$, $\Q D_{\Z _ {\ge 1},n}$ be their linear spans over $\Q$. Then we have
$$\calhd =\bigoplus _{n\ge 0} \Q W_{\Z_ {\ge 1},n}=\bigoplus _{n\ge 0} \Q D_{\Z_ {\ge 1},n},
$$
and the shuffle product is a graded product with respect to both gradings. More precisely,
$$\shap: \Q W_{\Z _ {\ge 1},n}\otimes \Q W_{\Z _ {\ge 1},m}\to \Q W_{\Z _ {\ge 1},n+m}
$$
and
$$\shap: \Q D_{\Z _ {\ge 1},n}\otimes \Q D_{\Z _ {\ge 1},m}\to \Q D_{\Z _ {\ge 1},n+m}.
$$

The grading by depth can be extended to $\calha$ easily: taking
$$ \ D_{\Z,0}:=\{{\bf 1}\},
\ D_{\Z ,n}:=\{[\vec s] \ | \ \vec s \in \Z ^n\}, \quad n\in \Z _{\ge 1},
$$
then
\begin {equation}
\mlabel {eq:dgrade}
\calha =\bigoplus _{n\ge 0} \Q D_{\Z ,n}.
\end{equation}

By the definition of $J$, we have
\begin {lemma}
\mlabel {lem:JStable}
For $n\in \Z _{\ge 1}$, there is $J: \Q D_{\Z ,n}\to \Q D_{\Z ,n}$, and $J|_{\Q D_{\Z ,n}}$ is invertible.
\end{lemma}

\begin {remark} The space $\calha$ also has a grading with respect to weight: take
$$W_{\Z,0}:=\{{\bf 1}, [\vec s] \ |\ \vec s \in \Z ^k, \ w([\vec s])=0\},
$$
and for a nonzero integer $n$, take
$$W_{\Z,n}:=\{[\vec s] \ |\ \vec s \in \Z ^k, \ w([\vec s])=n\}.
$$
Then
$$
\calha =\bigoplus _{n \in \Z} \Q W_{\Z ,n}.
$$
However, the operator $J$ does not behave so well with respect to this grading due to a subtlety at $1$.
\end{remark}

With all the preparations, we can state our main theorem in this section.
\begin{theorem}
\mlabel{thm:Xshap}
With respect to the depth grading in \meqref {eq:dgrade}, there is a unique graded associative product $\shap$ on $\calha$,  called the {\bf extended shuffle product}, which has the unit ${\bf 1}$ and such that
\begin{enumerate}
\item $[0]\shap [\vec s]=[0,\vec s]$ for $\vec s\in \Z^k$;
\item $[\vec s]\shap [0]=[0, \vec s]$ for $\vec s\in\Z_{>0}\times \Z^{k-1}$;
\item $J$ is a differential operator with respect to $\shap$.
\end{enumerate}
\mlabel {lem:Inverse}
\end{theorem}
\begin{proof} The proof is carried out in Section~\mref{ss:proof}.
	\end{proof}

Recall that a {\bf differential algebra} is an algebra $A$ equipped with a differential operator $d$. Thus by Theorem~\mref{lem:Inverse}, the triple $(\calha,\shap,J)$ is a differential algebra.

\subsection{The proof of Theorem~\mref{lem:Inverse}}

\mlabel{ss:proof}

The proof of Theorem~\mref{lem:Inverse} is divided into three parts, on the construction, the associativity and the uniqueness of $\shap$ respectively.

\subsubsection {The construction of $\shap$}
We now define the extended shuffle product. Based on the requirements of the product, the product for the elements of positive depth in the basis can be defined by a double recursion. To describe the recursion, we consider a partition of $\R^2$:
\begin{equation} \mlabel{eq:r2part}
	\R^2=R_1\sqcup R_2 \sqcup R_3 \sqcup R_4\sqcup R_5,
\end{equation}
where
\begin{align*}
	R_1:=&\{(s,t)\in \R^2\,|\, s=0\},\\
	R_2:=&\{(s,t)\in \R^2\,|\, s>0, t=0\}, \\
	R_3:=&\{(s,t)\in \R^2\,|\, s>0, t>0\}, \\
	R_4:=&\{(s,t)\in \R^2\,|\, s>0, t<0\}, \\
	R_5:=&\{(s,t)\in \R^2\,|\, s<0\}.
\end{align*}
The partition is shows as follows. Here the solid vertical line and horizontal ray are $R_1$ and $R_2$ respectively. The parallel dashed lines in each of the other three regions indicate the grading for the corresponding recursive in Definition~\mref{de:extsha}.

\begin{center}
	\begin{tikzpicture}
		\draw[ultra thick,-] (0,0) -- (1.75,0);\draw[dashed,-] (0,1) -- (1,0);\draw[dashed,-] (0,0.5) -- (0.5,0);\draw[dashed,-] (0,1.5) -- (1.5,0);
		\draw[ultra thick,-] (0,-1.75) -- (0,1.75); \draw[dashed,-] (0,-0.5) -- (1.5,-0.5);\draw[dashed,-] (0,-1) -- (1.5,-1);\draw[dashed,-] (0,-1.5) -- (1.5,-1.5);
		\draw[dashed,-] (-0.5,-1.5) -- (-0.5,1.5);\draw[dashed,-] (-1,-1.5) -- (-1,1.5);\draw[dashed,-] (-1.5,-1.5) -- (-1.5,1.5);
		\node at (0,1.95) {$R_1$};\node at (2.35,0) {$R_2$}; \node at (1.15,1) {$R_3$}; \node at (1,-0.75) {$R_4$};
		\node at (-1,0) {$R_5$};
	\end{tikzpicture}
\end{center}

\begin{defn} \mlabel{de:extsha}
We define a multiplication $\shap$ on $\calh_{\Z}$ by the following series of recursions.
First take ${\bf 1}$ to be the identity element. Next let $[s_1, \vec s\,']\in D_{\Z, n}, [t_1, \vec t\,']\in D_{\Z, m}$, with the convention that if $n=1$ or $m=1$, then take $\vec{s}\,'={\bf 1}$ or $\vec{t}\,'={\bf 1}$.

Depending on which $R_i$ the pair $(s_1,t_1)$ is in, define
$[s_1,\vec{s}\,']\shap [t_1,\vec{t}\,']$ by the corresponding recursion as follows. In the five recursions, the first two are carried out first. Then the last three use the first two as the initial (boundary) conditions. We take the first two cases as Step 1 and the last three cases as Step 2.

On the sum of depths $d([s_1, \vec s\,'])+d([t_1, \vec t\,'])=n+m$ of $[s_1, \vec s\,']$ and $[t_1, \vec t\,']$, which is divided into two cases:
\begin{enumerate}
  \item {\bf Case 1.} If $(s_1,t_1)$ is in $R_1$, that is, $s_1=0$, then use recursion on the sum of depths $d([s_1, \vec s\,'])+d([t_1, \vec t\,'])=n+m$ of $[s_1, \vec s\,']$ and $[t_1, \vec t\,']$ to define
  $$[0,\vec s\,']\shap [t_1, \vec t\,']:=[0, \vec s\,'\shap [t_1, \vec t\,']].
  $$
In particular,
  $$ [0]\shap [t_1]:=[0,t_1].$$
  \item {\bf Case 2.} If $(s_1,t_1)$ is in $R_2$, that is, $s_1>0$ and $t_1=0$, then use recursion on the sum of depths $d([s_1, \vec s\,'])+d([t_1, \vec t\,'])=n+m$ of $[s_1, \vec s\,']$ and $[t_1, \vec t\,']$ to define
  $$[s_1,\vec s\,']\shap [0, \vec t\,']:=[0, [s_1,\vec s\,']\shap \vec t\,'].
  $$
In particular,
$$ [s_1]\shap [0]:=[0,s_1], \quad s_1>0.$$
 \item {\bf Case 3.} If $(s_1,t_1)$ is in $R_3$, that is $s_1>0$ and $t_1>0$, then use recursion on $s_1+t_1$ to define
  $$[s_1,\vec s\,']\shap [t_1, \vec t\,']:=I([s_1,\vec s\,']\shap [t_1-1,\vec t\,'])+I([s_1-1,\vec s\,']\shap [t_1,\vec t\,']);
  $$
  \item {\bf Case 4.} If $(s_1,t_1)$ is in $R_4$, that is $s_1>0$ and $t_1<0$, then use recursion on $|t_1|$ to define
   $$[s_1,\vec s\,']\shap [t_1, \vec t\,']:=J([s_1,\vec s\,']\shap [t_1+1, \vec t\,'])-[s_1-1, \vec s\,']\shap [t_1+1, \vec t\,'];
   $$
  \item {\bf Case 5.} If $(s_1,t_1)$ is in $R_5$, that is $s_1<0$, then use recursion on $|s_1|$ to define
  $$[s_1,\vec s\,']\shap [t_1, \vec t\,']:=J([s_1+1,\vec s\,']\shap [t_1, \vec t\,'])-[s_1+1, \vec s\,']\shap [t_1-1, \vec t\,'].
  $$
\end{enumerate}
\end{defn}

Here are some examples, one for each of the five cases.
\begin{exam}
\begin{enumerate}[$(1)$]
  \item Case 1: $[0]\shap [-1]=[0,-1]$.
  \item Case 2: $[1]\shap [0]=[0,1]$.
  \item Case 3: $[1]\shap [1]=I([0]\shap [1]+[1]\shap [0])=2[1,1]$.
  \item Case 4: $[1]\shap [-1]=J([1]\shap [0])-[0]\shap [0]=[-1,1]-[0,0]$.
  \item Case 5: $[-1]\shap [0]=J([0]\shap [0])-[0]\shap [-1]=[-1, 0]-[0, -1]$.
\end{enumerate}
\end{exam}
Note that $[0]\shap [-1]\neq [-1]\shap [0]$. So the extended shuffle product defined in this way is not commutative.
\begin{remark}
When restricted to $\calib$,  this is the construction in \mcite{GZ1}. So we have the associativity in that case.
\end{remark}

\begin{defn}
A basis element $[s_1,\vec s\,']$ is called {\bf leading positive} if $s_1>0$, and an element in $\calia$ is called {\bf leading positive} if it is a linear combination of leading positive basis elements.
\end{defn}

By definition, for leading positive basis elements $[s_1,\vec s\,']$ and $[t_1,\vec t\,']$, we have
$$[s_1,\vec s\,']\shap [t_1,\vec t\,']=I([s_1,\vec s\,']\shap [t_1-1, \vec t\,']+[s_1-1,\vec s\,']\shap [t_1, \vec t\,']).
$$
So by a simple induction we conclude

\begin{lemma}
\mlabel{lem:positive}
For leading positive basis elements $[\vec s]$ and $[\vec t]$, the product $[\vec s]\shap [\vec t]$ is leading positive.
\end{lemma}

Again by the definition of $\shap$, we have
\begin{coro}
For leading positive basis elements $[\vec s]$ and $[\vec t]$, we have
$$([\vec s]\shap [\vec t] )\shap [0]=[0]\shap([\vec s]\shap [\vec t]).$$
\end{coro}

We first check that $J$ is a differential operator.
\begin{lemma}
For $\vec s\in \Z^m$ and $\vec t\in \Z^n$,
$$J([\vec s]\shap [\vec t])=J([\vec s])\shap [\vec t]+[\vec s]\shap J([\vec t]).
$$
So $J$ is a differential operator since $J({\bf 1})=0$.
\mlabel{lem:der}
\end{lemma}

\begin{proof}
Let $\vec s=[s_1, \vec s\,']$ and $\vec t=[t_1, \vec t']$. We prove the lemma in several cases.

{\bf Case 1:} If $s_1>0, t_1>0$, then
\begin{align*}
J\big([s_1, \vec s\,']\shap [t_1, \vec t\,']\big)&=J\circ I\big([s_1, \vec s\,']\shap J([t_1, \vec t\,'])+J([s_1, \vec s\,'])\shap [t_1, \vec t\,']\big)\\
&=[s_1, \vec s\,']\shap J([t_1, \vec t\,'])+J([s_1, \vec s\,'])\shap [t_1, \vec t\,'].
\end{align*}

{\bf Case 2:} If $s_1>0$, $t_1\leq 0$, then
 $$[s_1,\vec s\,']\shap J([t_1, \vec t\,'])=J\big([s_1,\vec s\,']\shap [t_1, \vec t\,']\big)-J([s_1,\vec s\,'])\shap [t_1,\vec t\,'],$$
that is
$$J\big([s_1,\vec s\,']\shap [t_1, \vec t\,']\big)=[s_1,\vec s\,']\shap J([t_1, \vec t\,'])+J([s_1,\vec s\,'])\shap [t_1,\vec t\,'].$$

{\bf Case 3:} If $s_1\leq 0$, then
$$J([s_1,\vec s\,'])\shap [t_1,\vec t\,']=J\big([s_1, \vec s\,']\shap [t_1, \vec t\,']\big)-[s_1,\vec s\,']\shap J([t_1,\vec t\,']),
$$
that is
$$\hspace{3cm} J\big([s_1,\vec s\,']\shap [t_1, \vec t\,']\big)=[s_1,\vec s\,']\shap J([t_1, \vec t\,'])+J([s_1,\vec s\,'])\shap [t_1,\vec t\,']. \hspace{3cm} \qedhere$$
\end{proof}

\subsubsection {The associativity} We now prove that the product defined in Definition~\mref {de:extsha} is associative for basis elements. This is obvious if one element is ${\bf 1}$. So we prove the associativity for basis elements of positive depth.
\begin{lemma}\mlabel{lem:111}
For $s,t,r\in \Z$,
$$([s]\shap[t])\shap [r]=[s]\shap([t]\shap [r]).
$$

\end{lemma}
\begin{proof}
Notice this is always true if $s=0$ by the definition. We prove the lemma in 4 cases, each with the previous case as the starting point.

{\bf Case 1:} If $s\ge 0$, $t\geq 0$, $r\geq 0$. This is proved in \mcite{GZ1}.

{\bf Case 2:} If $s\ge 0, t\geq 0, r\leq 0$. Let
$$Z_n=\{(s,t,r)\in \Z_{\ge 0}\times \Z_{\geq 0}\times \Z_{\leq 0} \ |\ r=-n\}.
$$
Then $\Z_{\ge 0}\times \Z_{\geq 0}\times \Z_{\leq 0}=\cup _0^\infty Z_n$. Denote
$$T:=\{(s,t,r)\in \Z_{\ge 0}\times \Z_{\geq 0}\times \Z_{\leq 0}\ |\ ([s]\shap[t])\shap [r]=[s]\shap ([t]\shap [r])\}.
$$
We now prove $Z_n\subset T$ by induction on $n\geq 0$. The initial step $Z_0\subset T$ is by Case 1.

Assume for $n\ge 0$, $Z_n\subset T$. Then consider $(s,t,r)\in Z_{n+1}$, which means $r=-(n+1)<0$. We have the following subcases.
\begin{enumerate}
\item[] {\bf Subcase 2.1:} When $s=0$, then $(0,t,r)\in T$ is always true.
\item [] {\bf Subcase 2.2:} When $s>0, t=0, r<0$, then since $[s]$ is leading positive, we have
\begin{align*}
([s]\shap [0])\shap [r]&=[0, s]\shap [r]=[0, [s]\shap  [r]]=[s]\shap [0, r]=[s]\shap ([0]\shap [r]).
\end{align*}
So $(s,0,r)$ is in $T$.
\item[] {\bf Subcase 2.3:} When $s>0, t>0, r<0$, then $(s, t, r+1), (s, t-1, r+1), (s-1, t, r+1)\in Z_n\subset T$. So
\begin{align*}
&([s]\shap [t])\shap [r]=([s]\shap [t])\shap J([r+1])\\
=&J(([s]\shap [t])\shap [r+1])-J([s]\shap [t])\shap[r+1]\\
=&J(([s]\shap [t])\shap [r+1])-([s]\shap [t-1]+[s-1]\shap [t])\shap [r+1],
\end{align*}
and
\begin{align*}
&[s]\shap ([t]\shap [r])=[s]\shap J([t]\shap [r+1])-[s]\shap ([t-1]\shap [r+1])\\
=&J([s]\shap ([t]\shap [r+1]))-[s]\shap ([t-1]\shap [r+1])-[s-1]\shap ([t]\shap[r+1]).
\end{align*}
This implies $(s,t,r)\in T$ by the induction hypothesis.
\end{enumerate}
So we always have $Z_{n+1}\subset T$. This implies that $T=\Z_{\ge 0}\times \Z_{\ge 0}\times \Z_{\le 0}$. So the lemma holds in this case.

{\bf Case 3:} If $s\ge 0$, $t\leq 0$. Define
$$Z_n:=\{(s,t,r)\in \Z_{\ge 0}\times \Z_{\leq 0}\times \Z\ |\ t=-n\}.
$$
Then
$\Z_{\ge 0}\times \Z_{\leq 0}\times \Z=\cup_{0}^{\infty}Z_n.$
Denote
$$T:=\{(s,t,r)\in \Z_{\geq 0}\times \Z_{\leq 0}\times \Z\ |\ ([s]\shap [t])\shap [r]=[s]\shap([t]\shap [r])\}.
$$
We now show $Z_n\subset T$ by induction on $n\geq 0$. As the initial step,
by Case 1 and 2, $Z_0$ is a subset of $T$. Assume $Z_n\subset T$ for $n \ge 0$ and consider $(s,t,r)\in Z_{n+1}$. So $t=-(n+1)<0$.
\begin{enumerate}
\item[] {\bf Subcase 3.1:} When $s=0$, then $(0, t, r)\in T$ is always true;
\item[] {\bf Subcase 3.2:} When $s>0, t<0$, then by Lemma \mref{lem:der}, we have
\begin{align*}
&([s]\shap [t])\shap [r]=J([s]\shap [t+1])\shap [r]-(J([s])\shap [t+1])\shap [r]\\
=&J(([s]\shap [t+1])\shap [r])-([s]\shap [t+1])\shap J([r])-(J([s])\shap [t+1])\shap [r]
\end{align*}
and
\begin{align*}
&[s]\shap ([t]\shap [r])=[s]\shap J([t+1]\shap [r])-[s]\shap ([t+1]\shap J([r]))\\
=&J([s]\shap ([t+1]\shap [r]))-J([s])\shap ([t+1]\shap [r])-[s]\shap ([t+1]\shap J([r])).
\end{align*}
\end{enumerate}
Now by the induction hypothesis, $(s,t,r)$ is in $T$. So $Z_{n+1}\subset T$. Hence $T=\Z_{\ge 0}\times \Z_{\leq 0}\times \Z$.

{\bf Case 4:} If $s\le  0$. Take
$$Z_n:=\{(s,t,r)\in \Z_{\le 0}\times \Z\times \Z\ |\ s=-n\}.
$$
Then
$\Z_{\le 0}\times \Z\times \Z=\cup_{n=0}^{\infty}Z_n.
$
Denote
$$T:=\big\{(s,t,r)\in \Z_{\leq 0}\times \Z\times \Z\ |\ ([s]\shap [t])\shap [r]=[s]\shap([t]\shap [r])\big\}.
$$
We prove $Z_n\subset T$ by induction on $n\geq 0$. We know for the start that $Z_0\subset T$ by the definition of the product. Assume $Z_n\subset T$ for $n\ge 0$. For $(s,t,r)\in Z_{n+1}$, there is $s=-(n+1)<0$.
By Lemma \mref{lem:der},
\begin{align*}
&\big([s]\shap [t]\big)\shap [r]=J\big([s+1]\shap [t]\big)\shap [r]-\big([s+1]\shap J([t])\big)\shap [r]\\
=&J\big(([s+1]\shap [t])\shap [r]\big)-\big([s+1]\shap [t]\big)\shap J([r])-\big([s+1]\shap J([t])\big)\shap [r],
\end{align*}
and
\begin{align*}
&[s]\shap\big([t]\shap [r]\big)=J\big([s+1]\shap([t]\shap [r])\big)-[s+1]\shap J\big([t]\shap [r]\big)\\
=&J\big([s+1]\shap([t]\shap [r])\big)-[s+1]\shap\big(J([t])\shap [r]\big)-[s+1]\shap\big([t]\shap J([r])\big),
\end{align*}
implying $Z_{n+1}\subset T$ by the induction hypothesis. Hence $T= \Z_{\leq 0}\times \Z\times \Z$. This finishes the proof of the lemma.
\end{proof}

\begin{prop}
\mlabel {pp:asso}
The product defined in Definition~\mref {de:extsha} is associative.
\end{prop}

\begin{proof} We only need to prove that for  $\vec s\in \Z^m, \vec t\in \Z^p, \vec r\in \Z^\ell$ with $m, p, \ell \in \Z_{>0}$, there is
\begin {equation}
\mlabel {eq:mnl}
([\vec s]\shap [\vec t])\shap [\vec r]=[\vec s]\shap ([\vec t]\shap [\vec r]).
\end{equation}
We prove Eq.~\meqref {eq:mnl} by a double induction, firstly, we do induction on $m+p+\ell\geq 3$. By Lemma \mref{lem:111}, this holds for $m=p=\ell=1$.  Assume that Eq.~\meqref{eq:mnl} holds for $m+p+\ell=k\ge 3$. Then for $m+p+\ell=k+1$,
let $(s_1,\vec s\,')\in \Z^m, (t_1,\vec t\,')\in \Z^p,  (r_1, \vec r')\in Z^\ell$.
Here we use the convention that if $[\vec s\,'], [\vec t\,']$ or $[\vec r\,']$ is of length zero, then it is taken to be ${\bf 1}$.

First notice that, if $s_1=0$, then the conclusion is true by the induction hypothesis:
\begin{align*}
&([0,\vec s\,']\shap [t_1,\vec t\,'])\shap [r_1,\vec r']=[0, \vec s\,'\shap [t_1,\vec t\,']]\shap [r_1, \vec r']=[0,(\vec s\,'\shap[t_1, \vec t\,'])\shap [r_1,\vec r']]\\
=&[0,\vec s\,'\shap ([t_1,\vec t\,']\shap [r_1, \vec r'])]=[0,\vec s\,']\shap ([t_1,\vec t\,']\shap [r_1,\vec r']).
\end{align*}

As in Lemma  \mref{lem:111}, we prove Eq. (\mref {eq:mnl}) by dividing into four cases.

{\bf Case 1:} If $s_1\geq 0$, $t_1\geq 0$, $r_1\geq 0$. Compared with the proof of Lemma \mref{lem:111}, this is new since it is not covered by \mcite{GZ1}. Let
$$Z_n:=\{(s_1,t_1,r_1)\in \Z_{\geq 0}\times \Z_{\geq 0}\times \Z_{\geq 0}\ |\ s_1+t_1+r_1=n\},
$$
$$T:=\{(s_1,t_1,r_1)\in\Z_{\geq 0}\times \Z_{\geq 0}\times \Z_{\geq 0}\ |\ ([s_1,\vec s\,']\shap [t_1,\vec t\,'])\shap [r_1,\vec r']=[s_1,\vec s\,']\shap ([t_1,\vec t\,']\shap [r_1,\vec r'])\}.
$$
We use induction on $n\geq 0$ to prove $Z_n \subset T$. Obviously $Z_0\subset T$ since it means $s_1=0$.
Now assume $Z_n\subset T$ for $n\ge 0$. Then for $(s_1,t_1,r_1)\in Z_{n+1}$,
we proceed as follows.

\noindent
{\bf Subcase 1.1} When $s_1=0$, we already have the conclusion,

\noindent
{\bf Subcase 1.2} When $s_1>0, t_1=0$, since $[s_1, \vec s\,']$ is leading positive, there is
\begin{align*}
&\big([s_1,\vec s\,']\shap [0,\vec t\,']\big)\shap [r_1,\vec r']=\big ([0,[s_1,\vec s\,']\shap [\vec t\,']\big )\shap [r_1,\vec r']=[0,\big ( [s_1,\vec s\,']\shap [\vec t\,']\big )\shap [r_1,\vec r']]\\
=&[0,[s_1,\vec s\,']\shap \big (  [\vec t\,']\shap [r_1,\vec r']\big )]=[s_1, \vec s\,']\shap [0, [\vec t\,']\shap [r_1, \vec r']]\\
=&[s_1, \vec s\,']\shap \big ([0,\vec t\,']\shap [r_1, \vec r']\big ).
\end{align*}

\noindent
{\bf Subcase 1.3} When $s_1>0, t_1>0, r_1=0$, since $[s_1, \vec s\,']\shap[t_1, \vec t\,']$ is a combination of leading positive terms, we have
\begin{align*}
&([s_1,\vec s\,']\shap [t_1,\vec t\,'])\shap [0,\vec r']=[0, ([s_1, \vec s\,']\shap [t_1, \vec t\,'])\shap [ \vec r']]=[0, [s_1, \vec s\,']\shap ([t_1, \vec t\,']\shap [\vec r'])]\\
=&[s_1, \vec s]\shap [0, [t_1, \vec t\,']\shap [\vec r']]=[s_1, \vec s\,']\shap ([t_1, \vec t\,']\shap [0, \vec r']).
\end{align*}

\noindent
{\bf Subcase 1.4} When $s_1>0, t_1>0, r_1>0$, we have
\begin{align*}
&([s_1,\vec s\,']\shap [t_1,\vec t\,'])\shap [r_1, \vec r']=I\big(([s_1,\vec s\,']\shap [t_1,\vec t\,'])\shap [r_1-1,\vec r']+J([s_1,\vec s\,']\shap [t_1,\vec t\,'])\shap [r_1,\vec r']\big)\\
=&I(([s_1,\vec s\,']\shap [t_1,\vec t\,'])\shap [r_1-1,\vec r'])+I((J([s_1,\vec s\,'])\shap [t_1,\vec t\,'])\shap [r_1,\vec r'])\\
&+I(([s_1,\vec s\,']\shap J([t_1,\vec t\,']))\shap [r_1,\vec r']).
\end{align*}
On the other hand,
\begin{align*}
&[s_1,\vec s\,']\shap ([t_1,\vec t\,']\shap [r_1,\vec r'])=I\big([s_1,\vec s\,']\shap J([t_1,\vec t\,']\shap [r_1,\vec r']))+[s_1-1,\vec s\,']\shap ([t_1,\vec t\,']\shap [r_1,\vec r'])\big)\\
=&I([s_1,\vec s\,']\shap (J([t_1,\vec t\,'])\shap [r_1,\vec r']))+I([s_1,\vec s\,']\shap ([t_1,\vec t\,']\shap J([r_1,\vec r'])))\\
&+I([s_1-1,\vec s\,']\shap ([t_1,\vec t\,']\shap [r_1,\vec r'])).
\end{align*}
So by the induction hypothesis, we have $(s_1,t_1,r_1)\in T$.

Hence $Z_{n+1}\subset T$. Therefore $T=\Z_{\geq 0}\times \Z_{\geq 0}\times \Z_{\geq 0}$, which means that Eq. (\mref {eq:mnl}) holds in Case 1.

{\bf Case 2:} If $s_1\geq 0$, $t_1\geq 0$, $r_1\leq 0$. Denote
$$Z_n:=\{(s_1,t_1,r_1)\in \Z_{\geq 0}\times \Z_{\geq 0}\times \Z_{\leq 0}\ |\ r_1=-n\},
$$
$$T:=\{(s_1,t_1,r_1)\in\Z_{\geq 0}\times \Z_{\geq 0}\times \Z_{\leq 0}\ |\ ([s_1,\vec s\,']\shap [t_1,\vec t\,'])\shap [r_1,\vec r']=[s_1,\vec s\,']\shap ([t_1,\vec t\,']\shap [r_1,\vec r'])\}.
$$
By Case 1, $Z_0\subset T$.
By the induction hypothesis on depth and a similar proof as for Lemma \mref{lem:111} Case 2, we have $T=\Z_{\ge 0}\times \Z_{\ge 0}\times \Z_{\le 0}$.

{\bf Case 3:} If $s_1\geq 0$, $t_1\leq 0$.
Denote
$$Z_n:=\{(s_1,t_1,r_1)\in \Z_{\geq 0}\times \Z_{\leq 0}\times \Z\ |\ t_1=-n\},
$$
$$T:=\{(s_1,t_1,r_1)\in \Z_{\geq 0}\times \Z_{\leq 0}\times \Z\ |\ ([s_1,\vec s\,']\shap [t_1,\vec t\,'])\shap [r_1,\vec r']=[s_1,\vec s\,']\shap ([t_1,\vec t\,']\shap [r_1,\vec r'])\}.
$$
Similar to the proof for Lemma \mref{lem:111} Case 3 and the induction hypothesis on depth, we have $T=\Z_{\geq 0}\times \Z_{\leq 0}\times \Z$.

{\bf Case 4:} If $s_1\leq 0$.  This case follows from the induction hypothesis on depth and a similar proof as in Lemma \mref{lem:111}  Case 4.

So we have finished the proof of Eq. (\mref {eq:mnl}) and thus of the associativity of the product in Definition~\mref {de:extsha}.
\end{proof}

By Lemma \ref{lem:JStable} and Definition~\mref {de:extsha}, we conclude

\begin {coro} The  product in Definition~\mref {de:extsha} is a graded product with respect to the grading in \meqref {eq:dgrade}.
\end{coro}

\subsubsection {Uniqueness}

Suppose that $\bar \shap$ is also a product on $\calha$ satisfying the conditions in Theorem \mref{thm:Xshap}. We will prove $[\vec s]\shap [\vec t]=[\vec s] \bar \shap [\vec t]$ on basis elements $[\vec s]$ and $[\vec t]$
by induction on the sum of depths. If one element is ${\bf 1}$, it is obvious. So we only need to prove it for basis elements with positive depth. We put the initial step into a lemma to serve as a prototype for other similar proofs later in the paper.

\begin{lemma}
\mlabel{lem:11}
For associative products $\shap$ and $\bar \shap$ in $\calha$ satisfying the conditions in Theorem \mref{thm:Xshap}, and any $s,t\in \Z$, we have
$$[s]\shap [t]=[s]\bar \shap [t].$$
\end{lemma}

\begin{proof} By the first condition of the products $\shap$ and $\bar \shap$ in Theorem~\mref{thm:Xshap}, this is true for $s=0$ and $t\in \Z$.
The rest of the proof is divided into three cases.

\noindent
{\bf Case 1. } If  $s\geq 0$, $t\geq 0$,  let
$$Z_n:=\{(s, t)\in \Z_{\ge 0}\times \Z_{\ge 0} \ \ | \ s+t=n\}
$$
$$T:=\{(s, t)\in \Z_{\ge 0}\times \Z_{\ge 0} \ | \ [s]\shap [t]=[s]\bar \shap [t]\}.$$ Obviously $Z_0\subset T$.
Assume $Z_n\subset T$. Then for $(s,t)\in Z_{n+1}$, we consider the following subcases.
\begin{enumerate}
\item[] {\bf Subcase 1.1:} When $s=0$ or $t=0$, we have
$[0]\shap [t]=[0, t]=[0]\bar \shap [t]$ by the first condition in Theorem~\mref{thm:Xshap} as mentioned above, and $[s]\shap [0]=[0, s]=[s]\bar \shap [0]$ by the second condition in Theorem~\mref{thm:Xshap}.
So $(0,t), (s,0)\in T$.
\item[] {\bf Subcase 1.2:} When $s\geq 1$ and $t\geq 1$, we have
$$
J([s]\shap [t])=[s]\shap [t-1]+[s-1]\shap [t]
=[s]\bar \shap [t-1]+[s-1]\bar \shap [t]=J([s]\bar \shap [t]).
$$
\end{enumerate}
Now both $[s]\shap [t]$ and $[s]\bar \shap [t]$ are in $\Q D_{\Z , 2}$. So Lemma \mref {lem:JStable} gives
$$[s]\shap [t]=[s]\bar \shap [t].$$
Thus $(s, t)$ is in $T$, which means $Z_{n+1}\subset T$. Thus $T=\Z_{\ge 0}\times \Z_{\ge 0}$.

\noindent
{\bf Case 2. } If $s\geq 0$, $t\leq 0$,
let
$$Z_n:=\{(s,t)\in \Z_{\geq 0}\times \Z_{\leq 0} \ \ | \ t=-n\},
$$
$$T:=\{(s,t)\in \Z_{ \geq 0}\times \Z_{\leq 0} \ | \ [s]\shap [t]=[s]\bar \shap [t]\}.$$
Then $Z_0\subset T$ by the first case.
Assume $Z_n\subset T$. Then for $(s,t)\in Z_{n+1}$, we know $t<0$. Then
\begin{enumerate}
\item[] {\bf Subcase 2.1:} When $s=0$, this is already covered at the beginning of the proof;
\item[] {\bf Subcase 2.2:} When $s\ge 1$, we have
\begin{align*}
&[s]\shap [t]=[s]\shap J([t+1])=J([s]\shap [t+1])-J([s])\shap [t+1]\\
=&J([s]\bar \shap [t+1])-J([s])\bar \shap [t+1]=[s]\bar \shap [t].
\end{align*}
So $(s, t)\in T$, and thus $Z_{n+1}\subset T$. So the induction shows $T=\Z_{\geq 0}\times \Z_{\leq 0}$.
\end{enumerate}

\noindent
{\bf Case 3. } If $s\leq 0$, let
$$Z_n:=\{(s,t)\in \Z_{\leq 0}\times \Z \ \ | \ s=-n\},
$$
$$T:=\{(s,t)\in \Z_{\leq 0}\times \Z \ | \ [s]\shap [t]=[s]\bar \shap [t]\}.$$
We have $Z_0\subset T$ by above cases. Assume that $Z_n\subset T$. Then for $(s,t)\in Z_{n+1}$, we have $s\leq -1$. So
\begin{align*}
&[s]\shap [t]=J([s+1])\shap [t]=J([s+1]\shap [t])-[s+1]\shap J([t])\\
=&J([s+1]\bar \shap [t])-[s+1]\bar \shap J([t])=[s]\bar \shap [t].
\end{align*}
So $(s, t)\in T$ and hence  $T=\Z_{\leq 0}\times \Z$. Thus we have the conclusion.
\end{proof}

\begin{prop}
\mlabel{pp:nn} The product that satisfies conditions in Theorem \mref {thm:Xshap} is unique.
\end{prop}

\begin{proof} If there are two products $\shap$ and $\bar \shap$ that  satisfy the conditions in Theorem \mref{thm:Xshap}, we only need prove that, for $\vec s\in \Z^m$, $\vec t\in \Z^n$ with $m,n\geq 1$, there is
\begin {equation}
\mlabel {eq:mn}
[\vec s]\shap [\vec t]=[\vec s]\bar \shap [\vec t].
\end{equation}

This is true for $m+n=2$ by Lemma \mref{lem:11}.  Assume Eq. (\mref {eq:mn}) holds for $m+n\leq k$. Let $[s_1,\vec s\,']\in \Z^{m}$, $[t_1, \vec t\,']\in \Z^{n}$ for $m+n=k+1$.  We can inductively prove $[s_1, \vec s\,']\shap [t_1, \vec t\,']=[s_1, \vec s\,']\bar \shap[t_1,\vec t\,']$ in a way  similar to the proof of Lemma \mref{lem:11} for $s_1$ and $t_1$.  This finishes the proof of uniqueness.
\end{proof}

\subsection {Subalgebras}

From  \mcite{GZ1},  we  know  that  $(\calhb, \shap )$  and  $(\calhd, \shap)$  are  two  subalgebras  of  $(\calha, \shap)$. There is another subalgebra which plays an important role in our subsequent work on the duality of Hopf algebras \mcite {GXZ}.

\begin{prop}
The triple $(\calhc, \shap, J)$  is  a  differential subalgebra  of  $(\calha, \shap, J)$.
\end{prop}

\begin{proof}
Obviously $\calhc$ is closed under the operator $J$. So we only need to prove that $\calhc$  is  closed  under $\shap$.
It is obvious that  for  $a\in \calhc$,
$$a\shap {\bf1}={\bf1}\shap a=a\in \calhc.$$

Now we use a double induction to prove that for  $\vec s\in \Z_{\le 0}^m$ and $\vec t\in \Z_{\le 0}^k$, the product $[\vec s]\shap [\vec t]$ is in $\calhc.$
We first use induction on the sum of depths $\ell:=d([\vec s])+d([\vec t])=m+k\geq 2$ of $[\vec s]$ and $[\vec t]$.

For the initial step of $\ell=2$, we have $m=k=1$. We apply induction on the absolute value $|s|$ to prove that $[s]\shap [t]\in \calhc$. The case for $s=0$ is obvious: $[0]\shap [t]=[0, t]\in\calhc$. Assume that the conclusion holds for $|s|=n>0$, that is $[-n]\shap [t]\in\calhc$. Then we consider $|s|=n+1$. By the induction  hypothesis and the stability of $J$ on $\calhc$, we have
\begin{align*}
[-n-1]\shap [t]=J([-n])\shap [t]=J([-n]\shap [t])-[-n]\shap J([t])\in\calhc.
\end{align*}

Now let $\ell\geq 2$ and assume that $[\vec s]\shap [\vec t]\in\calhc$ for $m+k=\ell\ge 2$ $(m>0,k>0)$. Then for $m+k=\ell+1$ $(m>0, k>0)$, denote $[\vec s]=[s_1, \vec s\,']$. We apply induction on $|s_1|$, by the induction hypothesis on depth, we have
\begin{align*}
[0, \vec s\,']\shap [\vec t]=[0, \vec s\,'\shap \vec t]\in\calhc.
\end{align*}
Assume that the conclusion holds for $|s_1|=n>0$, that is $[-n, \vec s\,']\shap [\vec t]\in\calhc$. Then we consider $|s_1|=n+1$. By the induction  hypothesis on $|s_1|=n$ and the stability of $J$ on $\calhc$, we have
\begin{align*}
[-n-1, \vec s\,']\shap [\vec t]=J([-n, \vec s\,'])\shap [\vec t]=J([-n, \vec s\,']\shap [\vec t])-[-n, \vec s\,']\shap J([\vec t])\in\calhc.
\end{align*}
Hence we have completed the proof.
\end{proof}

\section{The locality algebra of Chen symbols}
\mlabel{s:loc}
Now that we have defined the extended shuffle product, we want to show it is the structure underlying the multiple zeta series at convergent integer points, with arbitrary arguments. For this we need prove that the algebra corresponding to convergent integer points form a subalgebra, and taking the multiple zeta series defines an algebra homomorphism. To do this, we need to study the multiple zeta series in an abstract context.

\subsection{Locality algebras} Locality structures are important in exploring hidden structures. We first recall some backgrounds on locality structures from \mcite{CGPZ,CGPZ2,GPZ3}.

\begin{enumerate}
  \item  A locality set is a  couple $(X, \top)$ where $X$ is a set and
	$$ \top:= X\times_\top X \subseteq X\times X$$
	is a binary symmetric relation, which is called a {\bf locality relation} on $X$. For $x_1, x_2\in X$,
	denote $x_1\top x_2$ if $(x_1,x_2)\in \top$. For a subset $U\subset X$, the {\bf  {polar} subset} of $U$ is
\begin{equation*}
U^\top:= \{x\in X\,|\, (x, u)\in \top, \forall u\in U \}.  			
\end{equation*}
\item
For locality sets $(X,\top_X)$ and $(Y,\top_Y)$, $f, g : X\to Y$ is called {\bf mutually independent}, if
\begin{equation*} \mlabel{eq:locmap}
x_1\top_X x_2 \Longrightarrow f(x_1)\top_Y g(x_2), \quad \forall x_1, x_2\in X.
\end{equation*}
A map $f:X\to Y$ is called a {\bf locality map} if $f$ and $f$ are mutually independent.
  \item
A {\bf locality vector space} is a vector space $V$ equipped with a locality relation $\top$ which is compatible with the linear structure on $V$ in the sense that, for any  subset $X$ of $V$, $X^\top$ is a linear subspace of $V$.
\item A {\bf locality  algebra} over a 26 field $K$ is a locality vector space $(A,\top)$ over $K$ together with a map
$$ m_A: A\times_\top A \to A,
(u,v)\mapsto  u\cdot v=m_A(x,y)$$
which is locality bilinear in the sense that,
for $u, v, w\in A$ with $u\top w, v\top w$, then
$$ (u+v)\cdot w = u\cdot w + v\cdot w, \quad   w\cdot (u+v) = w\cdot u+w\cdot v, $$
$$ (ku)\cdot w =k(u\cdot w),\quad  u\cdot (kw)=k(u\cdot w), \ k\in K.$$
$m_A$ also satisfies the locality associativity: for $u, v, w\in A$ with $u\top v, u\top w, v\top w$, then
$$	
(u\cdot v) \top w,\quad  u\top (v\cdot w), \quad (u\cdot v) \cdot w = u\cdot (v\cdot w). 	
$$
\item A {\bf unitary locality algebra} is a locality algebra $(A,\top, m_A)$ with a unit $1_A$ such that $1_A\top u$ for each $u\in A$, and
$$1_A\cdot u=u\cdot 1_A=u.$$
\item Let $(A,m,\top_A)$ be a locality algebra. A locality subspace $(B, \top _B)$ of $A$ is called a {\bf locality subalgebra of $A$} if,
\vspace{-.2cm}	
$$\top_B:= \top_A \cap (B\times B)
\vspace{-.2cm}	
$$
and
$$m(\top _B)\subset B.
$$
\item A locality algebra $A$ with a grading $A=\oplus_{n\geq 0}A_n$ is called a {\bf locality graded algebra} if $m_A((A_m\times A_n)\cap \top_A) \subseteq A_{m+n}$ for all $m, n\in \Z _{\ge 0}$. The locality graded algebra is called {\bf connected} if $A_0=K$.
\item Given two locality algebras $(A_i, \top_i), i=1,2$,
a (resp. {\bf unitary}) {\bf locality algebra homomorphism} is a linear map   $\varphi:A_1\longrightarrow A_2$  such that
$a  \top_1 b$ implies $\varphi(a)\top_2\varphi(b)$ and
	$\varphi(a \cdot b)=\varphi(a )\cdot \varphi(b)$ for $a\top _1b$
(resp. and $\varphi(1_{A_1})=1_{A_2}$).
\end{enumerate}

\subsection{Generalized Chen fractions}
 Thanks to \mcite{GPZ2}, for the variables set
$\{x_i\}_{i\in \Z _{\ge 1}}^k, k\geq 1$, Chen fractions can be defined for any $\vec s \in \Z _{\ge 1}^k$. We generalize the definition to all $\vec s \in \Z ^k$.

\begin {defn}
\mlabel {defn:ChenF}
For a finite nonempty subset $K=\{i_1, \cdots, i_k\} \subset \Z _{\ge 1}$ with $k$ elements,  a {\bf generalized Chen fraction} with variables in  $\{x_j\ | j\in K\}$ is
a fraction of the form
$$\frac{1}{(x_{i_{1}}+x_{i_{2}}+\cdots+x_{i_{k}})^{s_1}(x_{i_{2}}+\cdots+x_{i_{k}})^{s_2}\cdots x_{i_{k}}^{s_k}},
$$
where $\vec s =(s_1, \cdots, s_k)\in \Z ^k$. We  also use the notation
$$f(\wvec{s_1, \cdots, s_k}{x_{i_{1}},\cdots, x_{i_{k}}})$$
to denote it.
Let $\fch$  be  the  set of  all generalized Chen fractions in any finite nonempty subset of $\{x_i\}_{i\in \Z _{\ge 1}}$, and $\Q \fch$ be the  subspace of  functions in variables  $\{x_i\}_{i\in \Z _{\ge 1}}$  that is linearly spanned by  $\fch$ over $\Q$.
\end{defn}

\begin {remark} By definition,
$$f(\wvec {0}{x_i})=1.
$$
Thus the elements of $\fch $ are not linearly independent. This is the key difference between generalized Chen fractions and Chen fractions in \mcite{GPZ2}.
\end {remark}

By definition, we have
\begin {lemma} There are recursive formulas for generalized Chen fractions:
\begin {equation}
\mlabel {eq:indchen}
\begin{split}
(x_{i_1}+\cdots +x_{i_k})f\wvec{s_1, \cdots, s_k}{x_{i_1}, \cdots, x_{i_k}}&=f\wvec{s_1-1, \cdots, s_k}{x_{i_1},\ \ \ \ \  \cdots, x_{i_k}},
\\
(x_{i_1}+\cdots +x_{i_k})^{-1}f\wvec{s_1, \cdots, s_k}{x_{i_1}, \cdots, x_{i_k}}&=f\wvec{s_1+1, \cdots, s_k}{x_{i_1},\ \ \ \ \  \cdots, x_{i_k}}.
\end{split}
\end{equation}
\end{lemma}

It is shown in \mcite {GPZ1} that the space $\calm $ of meromorphic germs with linear poles at zero with rational coefficients, together with the orthogonality locality relation $\perp $ induced by the standard inner product in $\R ^\infty$, is a locality algebra.
Since $\Q \fch$ is a subspace of $\calm $, the restriction  $\perp$   introduces a locality relation $\top $ on $\Q \fch$. This locality relation $\top $ can be described as follows: for two generalized Chen fractions  $g$ and $h$ in variables $\{x_i\ | i\in K\}$ and $\{x_j\ | j\in M\}$ with $K, M\subset \Z_{\ge 1}$, respectively, $g\top h$ if and only if $K\cap M=\emptyset$.

\begin{prop}
$(\Q \fch, \top, \cdot)$ is a locality algebra.
\end{prop}
\begin{proof} Since $(\Q \fch, \top )$ is a locality subspace of $\calm $, we only need to prove that $\Q \fch $ is closed under the locality function multiplication. For this purpose, it is enough to prove that the product of two elements $g, h$ in $\fch$ with $g\top h$ is in $\Q \fch $.

For $g=f(\wvec{s_1,\cdots, s_m}{x_{i_1}, \cdots, x_{i_m}}), h=f(\wvec{t_1, \cdots, t_p}{x_{j_1},\cdots, x_{j_p}})\in \fch$ with $g\top h$, which means
$$(M:=\{i_1, \cdots, i_m\})\cap (P:=\{j_1, \cdots, j_p\})=\emptyset,
$$
we now prove
$$
g\cdot h=\sum f(\wvec{u_1,\cdots, u_{m+p}}{x_{\ell _1}, \cdots, x_{\ell _{m+p}}}),
$$
where $\{\ell _1, \cdots, \ell_{m+p}\}$ is a permutation of $M \cup P$,
by induction on $d=m+p$ for $m>0, p>0$. For both the initial step $d=2$ and the induction step, following case-by-case study as in the proof of Lemma \mref {lem:11}, and using formulas in (\mref {eq:indchen}) to replace the operators $I$ and $J$, we have the conclusion.
\end{proof}

\subsection {Chen symbols}
The set $\fch$ is not a basis for $\Q \fch$. So to carry out the idea in \mcite {HXZ}, we introduce abstract Chen fractions--Chen symbols.

With the notations in Section~\mref{ss:notn}, for $X=\Z \times \Z_{\ge 1}$,  we have
$$H_{\Z \times \Z_{\ge 1}}^+:=\bigcup _{k\in \Z _{\ge 1}}(\Z \times \Z_{\ge 1})^k=\bigcup _{k\in \Z _{\ge 1}}\Z^k \times \Z_{\ge 1}^k, \quad \calh_{\Z \times \Z_{\ge 1}}^+ := \Q H_{\Z \times \Z_{\ge 1}}^+=  \bigoplus _{k\in \Z _{\ge 1}, \vec s\in \Z ^k,  \vec u\in \Z_{\ge 1}^k}\Q \wvec{\vec s}{\vec u}.$$
Here we use a double-row matrix
$\wvec {\vec s}{\vec u}$
to express an element in $H_{\Z \times \Z_{\ge 1}}^+$ with $\vec s \in \Z^k$ and $\vec u\in \Z _{\ge 1}^k$.

\begin {defn}  An element $\wvec {\vec s}{\vec u}$ is called a {\bf Chen symbol} if all entries of $\vec u$ are distinct. We will use $\sch$ to denote the set of Chen symbols together with ${\bf 1}$, and $\Q \sch$ the vector space with basis $\sch$.
\end{defn}

\begin {remark} By the definition, for a Chen symbol $\wvec{s_1, \cdots, s_k}{i_1, \cdots, i_k}$,
$f\wvec{s_1, \cdots, s_k}{x_{i_1}, \cdots, x_{i_k}}$ is a Chen fraction.
\end{remark}

  We now define a linear map
 \begin{equation} \mlabel{eq:fraci}
 I^S: \calh _{\Z \times \Z_{\ge 1}}^+\longrightarrow \calh _{\Z \times \Z_{\ge 1}}^+, \quad \wvec{\vec s}{\vec u}\mapsto \wvec{I(\vec s)}{\vec u},
\end{equation}
for basis elements $\wvec{\vec s}{\vec u}$. So its inverse is \begin{equation} \mlabel{eq:fracj}
J^S:\calh _{\Z \times \Z_{\ge 1}}^+\longrightarrow \calh _{\Z \times \Z_{\ge 1}}^+, \quad \wvec{\vec s}{\vec u}\mapsto \wvec{J(\vec s)}{\vec u}.
\end{equation}
By definition, once taking $J^S({\bf 1})=0$, then $\Q \sch$ is stable under the action of $J^S$.

\begin{defn} \mlabel{de:prodsymb}
We define a multiplication $\shap$ on
$$\calh_{\Z\times \Z_{\ge 1}}:=\calh_{\Z\times \Z_{\ge 1}}^+\oplus \Q {\bf 1}$$
in a similar manner to the extended shuffle product on $\calha$ defined in Definition~\mref {de:extsha}.

First we take ${\bf 1}$ as the unit. Then consider $\wvec {s_1,\vec s\,'}{u_1,\vec u'}$, $\wvec {t_1, \vec t\,'}{v_1,\vec v'}\in H_{\Z\times \Z_{\ge 1}}^+$ with $(s_1, \vec s\,')\in \Z^n$, $(u_1, \vec u')\in\Z_{\ge 1}^n$, $(t_1, \vec t\,')\in \Z^m$, $(v_1, \vec v')\in \Z_{\ge 1}^m$, with the convention that a depth zero vector is ${\bf 1}$. We define $\wvec {s_1,\vec s\,'}{u_1,\vec u'}\shap \wvec {t_1, \vec t\,'}{v_1,\vec v'} $ in five cases according to which of the five regions $R_i, 1\leq i\leq 5$ in Eq.~\meqref{eq:r2part} the pair $(s_1,t_1)$ belongs to.

\begin{enumerate}
\item If $(s_1,t_1)$ is in $R_1$, that is, $s_1=0$, then apply a recursion on the sum of depths  $d([s_1, \vec s\,'])+d([t_1, \vec t\,'])=n+m$ to define
 $$\wvec {0,\vec s\,'}{u_1,\vec u'}\shap \wvec {t_1, \vec t\,'}{v_1,\vec v'} :=\left [\wvec {0}{u_1}, \wvec {\vec s\,'}{\vec u'}\shap \wvec {t_1, \vec t\,'}{v_1,\vec v'}\right ];
  $$
\item If $(s_1,t_1)$ is in $R_2$, that is, $s_1>0, t_1=0$, then apply a recursion on the sum of depths  $d([s_1, \vec s\,'])+d([t_1, \vec t\,'])=n+m$ to define
  $$\wvec {s_1,\vec s\,'}{u_1,\vec u'}\shap \wvec {0, \vec t\,'}{v_1,\vec v'}:=\left[\wvec{0}{v_1}, \wvec{s_1,\vec s\,'}{u_1, \vec u'}\shap \wvec{\vec t\,'}{\vec v'}\right];
  $$
 \item If $(s_1,t_1)$ is in $R_3$, that is, $s_1>0, t_1>0$, then apply a recursion on $s_1+t_1$ to define
  $$\wvec {s_1,\vec s\,'}{u_1,\vec u'}\shap \wvec {t_1, \vec t\,'}{v_1,\vec v'} :=I^S\big(\wvec{s_1,\vec s\,'}{u_1, \vec u'}\shap \wvec{t_1-1,\vec t\,'}{\ \ v_1, \ \ \ \vec v'}\big)+I^S\big(\wvec{s_1-1,\vec s\,'}{\ \ u_1,\ \ \ \vec u'}\shap \wvec{t_1,\vec t\,'}{v_1, \vec v'}\big);
  $$
  \item If $(s_1,t_1)$ is in $R_4$, that is, $s_1>0, t_1<0$, then apply a recursion on $|t_1|$ to define
   $$\wvec {s_1,\vec s\,'}{u_1,\vec u'}\shap \wvec {t_1, \vec t\,'}{v_1,\vec v'} :=J^S\big(\wvec{s_1,\vec s\,'}{u_1, \vec u'}\shap \wvec{t_1+1, \vec t\,'}{\ \ v_1, \ \ \  \vec v'}\big)-\wvec{s_1-1, \vec s\,'}{\ \ u_1, \ \ \  \vec u'}\shap \wvec{t_1+1, \vec t\,'}{v_1,\ \ \ \ \  \vec v'};
   $$
  \item If $(s_1,t_1)$ is in $R_5$, that is, $s_1<0$, defined by induction on $|s_1|$,
  $$\wvec {s_1,\vec s\,'}{u_1,\vec u'}\shap \wvec {t_1, \vec t\,'}{v_1,\vec v'} :=J^S\big(\wvec{s_1+1,\vec s\,'}{\ \ u_1,\ \ \  \vec u'}\shap \wvec{t_1, \vec t\,'}{v_1, \vec v'}\big)-\wvec{s_1+1, \vec s\,'}{\ \ u_1,\ \ \  \vec u'}\shap \wvec{t_1-1, \vec t\,'}{\ \ v_1,\ \ \  \vec v'}.
  $$
\end{enumerate}
\end{defn}

Then just like the proof for the extended shuffle product on $\calha$ in Definition~\mref{de:extsha} , we can prove
\begin{prop} \mlabel{p:tworowprod}
Definition \mref {de:prodsymb} defines a noncommutative associative product on $\calh_{\Z \times \Z_{\ge 1}}$. Furthermore, in each basis element $\wvec{\vec{r}}{\vec{w}}$ that appears in the product $\wvec{\vec s}{\vec u}\shap \wvec{\vec t}{\vec v}$, the vector $\vec{w}$ is a shuffle of $\vec{u}$ and $\vec{v}$.
\end{prop}

Because of the similarity between the products on $\calh_{\Z \times \Z_{\ge 1}}$ and $\calh _\Z$ given in Definitions~ \mref {de:extsha} and \mref {de:prodsymb} respectively, the following result is easy to check.

\begin{prop} The linear map
$$\Phi: \calh_{\Z \times \Z_{\ge 1}}\to \calh _\Z, \quad \wvec {\vec s}{\vec u}\mapsto [\vec s]
$$
is an algebra homomorphism.
\mlabel{prop:phihomo}
\end{prop}

Now we introduce a locality relation $\top$ on $\sch$: for $\wvec{\vec s}{\vec u}, \wvec{\vec t}{\vec v}\in \sch$, define
$${\bf 1}\top  \wvec{\vec s}{\vec u},
$$
$$ \wvec{\vec s}{\vec u}\top \wvec{\vec t}{\vec v} \Longleftrightarrow {\rm entries \ of } \ \vec u, \vec v \ {\rm are \ distinct},$$
and extend it to a locality relation on $\Q \sch$.

\begin{prop}
\mlabel {prop:zhomo}
\begin{enumerate}
\item The triple $(\Q \sch, \top, \shap)$ is a locality algebra.
\mlabel{it:zhomo1}
\item  \mlabel{it:zhomo2}
The linear map
\begin{equation} \mlabel{eq:fmap}
F: \Q \sch \longrightarrow \Q \fch, \quad \wvec{s_1,\cdots, s_k}{i_1,\cdots, i_k}\mapsto f\wvec{s_1,\cdots, s_k}{x_{i_1},\cdots, x_{i_k}}
\end{equation}
is a locality algebra homomorphism, that is,
$$F(\wvec{\vec s}{\vec u}\shap\wvec{\vec t}{\vec v})=F(\wvec{\vec s}{\vec u})F(\wvec{\vec t}{\vec v})$$
for $\wvec{\vec s}{\vec u}\top  \wvec{\vec t}{\vec v} \in \sch$.
\end{enumerate}
\end{prop}

\begin{proof}
\meqref{it:zhomo1} This follows from Proposition~\mref{p:tworowprod}.

\meqref{it:zhomo2} We only need to prove basis elements of $\sch$, it is obvious that
$$F({\bf1}\shap \wvec{\vec s}{\vec u})=F({\bf 1}) F(\wvec{\vec s}{\vec u})=F(\wvec{\vec s}{\vec u})=F(\wvec{\vec s}{\vec u})F({\bf 1})=F(\wvec{\vec s}{\vec u}\shap {\bf 1}).
$$
Now for $\wvec{\vec s}{\vec u}, \wvec{\vec t}{\vec v} \in \sch$, let ${\vec s}\in \Z^m, {\vec u}\in\Z_{\ge 1}^m, \vec t\in \Z^p, \vec v\in \Z_{\ge 1}^p$, we prove the conclusion by a double induction. Firstly, we do induction on the sum depth $d:=m+p$.
~\\
{\bf Step 1:} For $d=2$, we check case by case.
\begin{enumerate}
 \item {\bf Case 1:}If $s=0$, then we know
$$F(\wvec{0}{i_1}\shap \wvec{t}{i_2})=F(\wvec{0, t}{i_1, i_2})=\frac{1}{x_{i_2}^{t}}=F(\wvec{0}{i_1})F(\wvec{t}{i_2}).
$$
 \item {\bf Case 2:} If $s\ge 0, t\ge 0$,
   let
  $$Z_n:=\{(s, t)\in \Z_{\ge 0}\times \Z_{\ge 0} | s+t=n\},
  $$
  and
  $$T:=\{(s, t)\in \Z_{\ge 0}\times \Z_{\ge 0} | F(\wvec{s}{i_1}\shap \wvec{t}{i_2})=F(\wvec{s}{i_1})F(\wvec{t}{i_2})\}.
  $$
If $s=0$ or $t=0$, then we have already known $F(\wvec{s}{i_1}\shap \wvec{t}{i_2})=F(\wvec{s}{i_1})F(\wvec{t}{i_2})$, so $Z_0\subset T$. Assuming $Z_n\subset T$, then for $(s, t)\in Z_{n+1}$, we have
{\small
  \begin{equation*}
  \begin{split}
  F(\wvec{s}{i_1}\shap \wvec{t}{i_2})&=F\big(I^S(\wvec{s}{i_1}\shap \wvec{t-1}{i_2}+\wvec{s-1}{i_1}\shap \wvec{t}{i_2})\big)\\
  &=\frac{1}{x_{i_1}+x_{i_2}}\big(F(\wvec{s}{i_1}\shap \wvec{t-1}{i_2})+F(\wvec{s-1}{i_1}\shap \wvec{t}{i_2})\big)\\
  &=\frac{1}{x_{i_1}+x_{i_2}}\big(F(\wvec{s}{i_1})F(\wvec{t-1}{i_2})+F(\wvec{s-1}{i_1})F(\wvec{t}{i_2})\big)\\
  &=\frac{1}{x_{i_1}+x_{i_2}}(\frac{1}{x_{i_1}^s x_{i_2}^{t-1}}+\frac{1}{x_{i_1}^{s-1}x_{i_2}^t})\\
  &=\frac{1}{x_{i_1}^s x_{i_2}^{t}}=F(\wvec{s}{i_1})F(\wvec{t}{i_2}).
  \end{split}
  \end{equation*}
}
  Hence $Z_{n+1}\subset T$.
  \item {\bf Case 3:} If $s\ge 0, t\le0$,  let
  $$Z_n:=\{(s, t)\in \Z_{\ge 0}\times \Z_{\le 0} | t=-n\},
  $$
  and
  $$T:=\{(s, t)\in \Z_{\ge 0}\times \Z_{\le 0} | F(\wvec{s}{i_1}\shap \wvec{t}{i_2})=F(\wvec{s}{i_1})F(\wvec{t}{i_2})\},
  $$
 We know $(0, t)\in T$, and then for $(s, 0)\in Z_0$,
\begin{equation*}
\begin{split}
&F(\wvec{s}{i_1}\shap\wvec{0}{i_2})=F(\wvec{0,s}{i_2, i_1})=f\wvec{0,s}{x_{i_2}, x_{i_1}}=\frac{1}{x_{i_1}^s}=f\wvec{0}{x_{i_2}}f\wvec{s}{x_{i_1}}\\
=&F(\wvec{s}{i_1})F(\wvec{0}{i_2}).
\end{split}
\end{equation*}
So $Z_0\subset T$.
  Now assume $Z_n\subset T$, for $(s, t)\in Z_{n+1}$, we only need to consider $s>0$. Then
{\small
\begin{equation*}
\begin{split}
F(\wvec{s}{i_1}\shap \wvec{-n-1}{i_2})&= F\big(\wvec{s}{i_1}\shap J^S(\wvec{-n}{i_2})\big)\\
 &=F\big(J^S(\wvec{s}{i_1}\shap\wvec{-n}{i_2})\big)-F\big(J^S(\wvec{s}{i_1})\shap \wvec{-n}{i_2}\big)\\
 &=(x_{i_1}+x_{i_2})f\wvec{s}{x_{i_1}}f\wvec{-n}{x_{i_2}}-f\wvec{s-1}{x_{i_1}} f\wvec{-n}{x_{i_2}}\\
 &=(x_{i_1}+x_{i_2})\frac{x_{i_2}^n}{x_{i_1}^s}-\frac{x_{i_2}^n}{x_{i_1}^{s-1}}\\
 &=\frac{x_{i_2}^{n+1}}{x_{i_1}^s}=f\wvec{s}{x_{i_1}}f\wvec{-n-1}{x_{i_2}}=F(\wvec{s}{i_1})F(\wvec{-n-1}{i_2}).
 \end{split}
\end{equation*}
}
Therefore, $(s, -n-1)\in T$ and hence $Z_{n+1}\subset T$.
  \item {\bf Case 4:} If $s\le 0$,
 let
 $$Z_n:=\{(s, t)\in \Z_{\le 0}\times \Z | s=-n\},
 $$
 and
 $$T:=\{(s, t)\in \Z_{\le 0}\times \Z |  F(\wvec{s}{i_1}\shap \wvec{t}{i_2})= F(\wvec{s}{i_1})F(\wvec{t}{i_2})\}.
 $$
We know $Z_0\subset  T$ by the previous case 1. Now assume $Z_n\subset T$, and take $(s, t)\in Z_{n+1}$. Then we have
{\small
 \begin{equation*}
 \begin{split}
 F(\wvec{-n-1}{i_1}\shap \wvec{t}{i_2})&= F\big(J^S(\wvec{-n}{i_1})\shap \wvec{t}{i_2}\big)\\
 &=F\big(J^S(\wvec{-n}{i_1}\shap\wvec{t}{i_2})\big)-F\big(\wvec{-n}{i_1}\shap J^S(\wvec{t}{i_2})\big)\\
 &=(x_{i_1}+x_{i_2})f\wvec{-n}{x_{i_1}}f\wvec{t}{x_{i_2}}-f\wvec{-n}{x_{i_1}} f\wvec{t-1}{x_{i_2}}\\
 &=(x_{i_1}+x_{i_2})x_{i_1}^n\frac{1}{x_{i_2}^t}-x_{i_1}^n\frac{1}{x_{i_2}^{t-1}}\\
 &=\frac{x_{i_1}^{n+1}}{x_{i_2}^t}=f\wvec{-n-1}{x_{i_1}}f\wvec{t}{x_{i_2}}=F(\wvec{-n-1}{i_1})F(\wvec{t}{i_2}),
 \end{split}
 \end{equation*}
}
implying $(-n-1, t)\in T$, so $Z_{n+1}\subset T$.
\end{enumerate}
{\bf Step2:} Then assume that the conclusion holds for $d=m+p=\ell \ge 2$, that is,
$$F(\wvec{s_1,\cdots,s_m}{i_1,\cdots,i_m}\shap\wvec{t_1, \cdots, t_p}{i_{m+1}, \cdots, i_{m+p}})=F(\wvec{s_1,\cdots,s_m}{i_1,\cdots,i_m})F(\wvec{t_1,\cdots, t_p}{i_{m+1}, \cdots, i_{m+p}})
$$
for $d=\ell $, with $i_1, \cdots, i_{m+p}$ distinct. Then for $d=m+p=\ell+1$, we proceed in four cases.
\begin{enumerate}
  \item {\bf Case 1:}If $s_1=0$, then
{\small
\begin{equation*}
\begin{split}
&F(\wvec{0,\cdots,s_m}{i_1,\cdots,i_m}\shap\wvec{t_1, \cdots, t_p}{i_{m+1}, \cdots, i_{m+p}})\\
=&F(\big[\wvec{0}{i_1}, \wvec{s_2,\cdots,s_m}{i_2,\cdots, i_m}\shap\wvec{t_1,\cdots, t_p}{i_{m+1}, \cdots, i_{m+p}}\big])=f(\wvec{s_2,\cdots,s_m}{x_{i_2},\cdots, x_{i_m}}\shap\wvec{t_1,\cdots, t_p}{x_{i_{m+1}}, \cdots, x_{i_{m+p}}})\\
=&f\wvec{s_2,\cdots,s_m}{x_{i_2},\cdots, x_{i_m}} f\wvec{t_1,\cdots, t_p}{x_{i_{m+1}}, \cdots, x_{i_{m+p}}}=f\wvec{0,s_2,\cdots,s_m}{x_{i_1}, x_{i_2},\cdots, x_{i_m}} f\wvec{t_1,\cdots, t_p}{x_{i_{m+1}}, \cdots, x_{i_{m+p}}}\\
=&F(\wvec{0,s_2,\cdots,s_m}{i_1,\cdots,i_m})F(\wvec{t_1,\cdots, t_p}{i_{m+1}, \cdots, i_{m+p}}),
\end{split}
\end{equation*}
}
where the third equality follows from the induction hypothesis.
  \item {\bf Case 2:}If $s_1\ge 0, t_1\ge 0$, let
  $$Z_n:=\{(s_1, t_1)\in \Z_{\ge 0}\times \Z_{\ge 0} | s_1+t_1=n\},
  $$
  and $$T:=\{(s_1, t_1)\in \Z_{\ge 0}\times \Z_{\ge 0} | F(\wvec{\vec s}{\vec u}\shap \wvec{\vec t}{\vec v})=F(\wvec{\vec s}{\vec u})F(\wvec{\vec t}{\vec v})\}.
  $$
We know $(0,t_1)\in T$ fr any $t_1\in \Z$ and $(s_1,0)\in T$ for $s_1\in \Z _{\ge 0}$. Hence $Z_0\subset T$. Now assuming $Z_n\subset T$, for $(s_1, t_1)\in Z_{n+1}$, we can assume $s_1>0$ and $t_1>0$. Then
\begin{align*}
 &F(\wvec{s_1, \cdots, s_m}{i_1, \cdots, i_m}\shap \wvec{t_1, \cdots, t_p}{i_{m+1}, \cdots, i_{m+p}})\\
=&F\big(I^S(\wvec{s_1,\cdots, s_m}{i_1, \cdots, i_m}\shap \wvec{t_1-1,\cdots, t_p}{i_{m+1}, \cdots, i_{m+p}})+I^S(\wvec{s_1-1,\cdots, s_m}{i_1, \cdots, i_m}\shap \wvec{t_1,\cdots, t_p}{i_{m+1}, \cdots, i_{m+p}})\big)\\
=&\frac{1}{x_{i_1}+\cdots+x_{i_{m+p}}}F(\wvec{s_1,\cdots, s_m}{i_1, \cdots, i_m}\shap \wvec{t_1-1,\cdots, t_p}{i_{m+1}, \cdots, i_{m+p}}+\wvec{s_1-1,\cdots, s_m}{i_1, \cdots, i_m}\shap \wvec{t_1,\cdots, t_p}{i_{m+1}, \cdots, i_{m+p}})\\
=&\frac{1}{x_{i_1}+\cdots+x_{i_{m+p}}}\big(F(\wvec{s_1,\cdots, s_m}{i_1, \cdots, i_m})F(\wvec{t_1-1,\cdots, t_p}{i_{m+1}, \cdots, i_{m+p}})+F(\wvec{s_1-1,\cdots, s_m}{i_1, \cdots, i_m})F(\wvec{t_1,\cdots, t_p}{i_{m+1}, \cdots, i_{m+p}})\big)\\
=&\frac{x_{i_{m+1}}+\cdots+x_{i_{m+p}}}{(x_{i_1}+\cdots+x_{i_{m+p}})}F(\wvec{s_1,\cdots, s_m}{x_{i_1}, \cdots, x_{i_m}})F(\wvec{t_1,\cdots, t_p}{x_{i_{m+1}}, \cdots, x_{i_{m+p}}})\\
&+\frac{x_{i_{1}}+\cdots+x_{i_{m}}}{(x_{i_1}+\cdots+x_{i_{m+p}})}F(\wvec{s_1,\cdots, s_m}{x_{i_1}, \cdots, x_{i_m}})F(\wvec{t_1,\cdots, t_p}{x_{i_{m+1}}, \cdots, x_{i_{m+p}}})\\
=&F(\wvec{s_1,\cdots, s_m}{i_1, \cdots, i_m})F(\wvec{t_1,\cdots, t_p}{i_{m+1}, \cdots, i_{m+p}}).
\end{align*}
Thus $Z_{n+1}\subset T$.
  \item {\bf Case 3:} If $s_1\ge 0, t_1\le 0$,
  let
  $$Z_n:=\{(s_1, t_1)\in \Z_{\ge 0}\times \Z_{\le 0} | t_1=-n\}
  $$
  $$T:=\{(s_1, t_1)\in \Z_{\ge 0}\times \Z_{\le 0} | F(\wvec{\vec s}{\vec u}\shap \wvec{\vec t}{\vec v})=F(\wvec{\vec s}{\vec u})F(\wvec{\vec t}{\vec v})\}.
  $$
We know $(0,t_1)\in T$. So for $(s_1, 0)\in Z_0$, we might assume $s_1>0$.
Since
\begin{equation*}
\begin{split}
&F(\wvec{s_1,\cdots,s_m}{i_1,\cdots, i_m}\shap \wvec{0,\cdots, t_p}{i_{m+1},\cdots, i_{m+p}})=F(\big[\wvec{0}{i_{m+1}}, \wvec{s_1,\cdots,s_m}{i_1,\cdots, i_m}\shap \wvec{t_2,\cdots, t_p}{i_{m+2},\cdots, i_{m+p}}\big])\\
=&F(\wvec{s_1,\cdots,s_m}{i_1,\cdots, i_m}\shap \wvec{t_2,\cdots, t_p}{i_{m+2},\cdots, i_{m+p}})=F(\wvec{s_1,\cdots,s_m}{i_1,\cdots, i_m})F( \wvec{t_2,\cdots, t_p}{i_{m+2},\cdots, i_{m+p}})\\
=&F(\wvec{s_1,\cdots,s_m}{i_1,\cdots, i_m}\shap \wvec{0,t_2,\cdots, t_p}{i_{m+1},i_{m+2}\cdots, i_{m+p}}),
\end{split}
\end{equation*}
we have $Z_0\subset T$. Now assuming $Z_n\subset T$, for $(s_1, t_1)\in Z_{n+1}$, we can assume $s_1>0$. Then
\begin{equation*}
\begin{split}
 &F(\wvec{s_1,\cdots, s_m}{i_1, \cdots,i_m}\shap\wvec{-n-1,\cdots,t_p}{i_{m+1},\cdots, i_{m+p}})=F\big(\wvec{s_1,\cdots, s_m}{i_1, \cdots,i_m}\shap J^S(\wvec{-n,\cdots,t_p}{i_{m+1},\cdots, i_{m+p}})\big)\\
=&F\big(J^S(\wvec{s_1,\cdots, s_m}{i_1, \cdots,i_m}\shap\wvec{-n,\cdots,t_p}{i_{m+1},\cdots, i_{m+p}})\big)-F\big(J^S(\wvec{s_1,\cdots, s_m}{i_1, \cdots,i_m})\shap \wvec{-n,\cdots,t_p}{i_{m+1},\cdots, i_{m+p}}\big)\\
=&(x_{i_1}+\cdots+x_{i_{m+p}})F\big(\wvec{s_1,\cdots, s_m}{i_1, \cdots,i_m})F(\wvec{-n,\cdots,t_p}{i_{m+1},\cdots, i_{m+p}}\big)\\
&-F\big(\wvec{s_1-1,\cdots, s_m}{i_1, \cdots,i_m})F(\wvec{-n,\cdots,t_p}{i_{m+1},\cdots, i_{m+p}}\big)\\
=&F(\wvec{s_1,\cdots, s_m}{i_1, \cdots,i_m})F(\wvec{-n-1,\cdots,t_p}{i_{m+1},\cdots, i_{m+p}}).
\end{split}
\end{equation*}
Thus $(s_1, -n-1)\in T$, which means $Z_{n+1}\subset T$.
  \item {\bf Case 4:}If $s_1\le 0$, let
  $$Z_n:=\{(s_1, t_1)\in \Z_{\le 0}\times \Z | s_1=-n\}
  $$
  and
  $$T:=\{(s_1, t_1)\in \Z_{\le 0}\times \Z | F(\wvec{\vec s}{\vec u}\shap \wvec{\vec t}{\vec v})=F(\wvec{\vec s}{\vec u})F(\wvec{\vec t}{\vec v})\}.
  $$
We know $(0, t_1)\in T$, and so $Z_0\subset T$. Assuming $Z_n\subset T$, then for $(-n-1, t_1)\in Z_{n+1}$, we have
{\small
\begin{equation*}
\begin{split}
 &F(\wvec{-n-1, \cdots, s_m}{i_1,\cdots, i_m}\shap \wvec{t_1, \cdots, t_p}{i_{m+1}, \cdots, i_{m+p}})=F\big(J^S(\wvec{-n, \cdots, s_m}{i_1,\cdots, i_m})\shap \wvec{t_1, \cdots, t_p}{i_{m+1}, \cdots, i_{m+p}}\big)\\
=&F\big(J^S(\wvec{-n, \cdots, s_m}{i_1,\cdots, i_m}\shap \wvec{t_1, \cdots, t_p}{i_{m+1}, \cdots, i_{m+p}})\big)-F(\wvec{-n, \cdots, s_m}{i_1,\cdots, i_m}\shap J^S(\wvec{t_1, \cdots, t_p}{i_{m+1}, \cdots, i_{m+p}}))\\
=&(x_{i_1}+\cdots+x_{i_{m+p}})F(\wvec{-n, \cdots, s_m}{i_1,\cdots, i_m})F( \wvec{t_1, \cdots, t_p}{i_{m+1}, \cdots, i_{m+p}})\\
&-F(\wvec{-n, \cdots, s_m}{i_1,\cdots, i_m})F(\wvec{t_1-1, \cdots, t_p}{i_{m+1}, \cdots, i_{m+p}})\\
=&(x_{i_1}+\cdots+x_{i_{m}})F(\wvec{-n, \cdots, s_m}{i_1,\cdots, i_m})F( \wvec{t_1, \cdots, t_p}{i_{m+1}, \cdots, i_{m+p}})\\
&+(x_{i_{m+1}}+\cdots+x_{i_{m+p}})F(\wvec{-n, \cdots, s_m}{i_1,\cdots, i_m})F( \wvec{t_1, \cdots, t_p}{i_{m+1}, \cdots, i_{m+p}})\\
&-F(\wvec{-n, \cdots, s_m}{i_1,\cdots, i_m})F(\wvec{t_1-1, \cdots, t_p}{i_{m+1}, \cdots, i_{m+p}})\\
=&F(\wvec{-n-1, \cdots, s_m}{i_1,\cdots, i_m})F(\wvec{t_1, \cdots, t_p}{i_{m+1}, \cdots, i_{m+p}})+F(\wvec{-n, \cdots, s_m}{i_1,\cdots, i_m})F(\wvec{t_1-1, \cdots, t_p}{i_{m+1}, \cdots, i_{m+p}})\\
&-F(\wvec{-n, \cdots, s_m}{i_1,\cdots, i_m})F(\wvec{t_1-1, \cdots, t_p}{i_{m+1}, \cdots, i_{m+p}})\\
=&F(\wvec{-n-1, \cdots, s_m}{i_1,\cdots, i_m})F(\wvec{t_1, \cdots, t_p}{i_{m+1}, \cdots, i_{m+p}}).
\end{split}
\end{equation*}
}
Thus $(-n-1, t_1)\in T$, which implies $Z_{n+1}\subset T$.
\end{enumerate}
In summary, we have completed the inductive proof, whence $F$ is a locality algebra homomorphism.
\end{proof}

\section{The subalgebra for convergent series}
\mlabel{s:conv}

We know that $\calh^0$, the space spanned by all convergent positive integer arguments, is a subalgebra of $\calh_{\Z_{\ge 1}}$ under the shuffle product and that
$$\zeta^\shap: (\calh^0,  \shap)\longrightarrow \R$$
is an algebra homomorphism. This will be extended to the space spanned by all convergent integer arguments for multiple zeta series.

For $\vec s\in \Z ^k$, the multiple zeta series
converges when
$$s_1+\cdots+ s_j>j, \ j=1, \cdots, k.$$
Denote
$$\calh ^0_\Z: =\Q{\bf 1}\oplus\bigoplus_{\vec s\in \Z^k, s_1+\cdots +s_j>j, j=1, \cdots, k, k\ge 1} \Q[\vec s]$$
be the linear span of these convergent integer points. We first prove that it is a subalgebra of $(\calha , \shap)$.

\subsection{Estimation of partial weights} Since the criterion of convergency involves partial weights, to show that $\calh ^0_\Z$ is a subalgebra, we need to estimate the partial weights of terms appearing in the extended shuffle product.

  For $\vec s=(s_1, \cdots, s_m) \in \Z ^m$, recall from Section ~\mref{ss:notn}, that $w_i([\vec s]):=s_1+\cdots+s_i$ of $[\vec s]$  is called the $i$-th partial weight. Because of the possible appearance of negative entries, we have to modify the notion of partial weights.
For $\vec s \in \Z ^m$, denote
\begin{equation} \mlabel{eq:tildepw}
\tilde w_j([\vec s]):=\left \{ \begin {array}{ll}w_j([\vec s]),& \text{if either } s_{j+1}> 0\, \text{ or }\, j=m,\\
w_{j+1}([\vec s]),& \text{if } s_{j+1}\le  0,
\end{array}\right .
\end{equation}
and for ${\bf 1}$, let
$$w_i({\bf 1}):=\tilde w_i({\bf 1}):=0, \ i\ge 0.
$$

\begin{exam}
	$$\tilde w_0([-1,2])=-1, \tilde w_1([-1,2])=-1, \tilde w_2([-1, 2])=1;
	$$
	$$\tilde w_0([2,-1])=0, \tilde w_1([2,-1])=1, \tilde w_2([2, -1])=1;
	$$
\end{exam}

\begin {remark}\label{rem:wei>pwei} By definition, $w_j([\vec s])\ge \tilde w_j([\vec s])$, and $w_j([\vec s])= \tilde w_j([\vec s])$ if and only if either $s_{j+1}\ge 0$ or $j=m$. Note that when $j=m$, $s_{j+1}$ is not defined. Also we have $\tilde w_0([\vec s])\le 0$.
\end{remark}

Furthermore, the following result also follows from the definition.
\begin{lemma}
\mlabel{eq:kandk-1}
For $k\ge 1$,
$$\tilde w_k([0,\vec s])=\tilde w_{k-1}([\vec s]).
$$
\end{lemma}
  As we know, the extended shuffle product keeps depths, that is, for $\vec s \in \Z^m$ and $\vec t\in \Z^p$, then the depth of every term of $[\vec s]\shap [\vec t]$ equals to $m+p$.

\begin{prop}
\mlabel{lem:termgemin} For basis elements $[\vec s]$ and $[\vec t]$ with $\vec s \in \Z^m$ and $\vec t\in \Z^p$, if for $\vec{v}\in \Z^{m+p}$, the symbol $[\vec{v}]$ appears in
$$[\vec s]\shap [\vec t]=\sum _{\vec v}A_{\vec s,\vec t}^{\vec v} [\vec v]
$$
with $A_{\vec s,\vec t}^{\vec v}\not =0$, then for $1\le k\le m+p=d([\vec v])$, there is the lower bound of partial weights of $[\vec{v}]:$
\begin{equation}
\mlabel{eq:convergentmin}
w_k([\vec v])\ge {\rm min}\{\tilde w_i([\vec s])+\tilde w_j([\vec t])\}_{0\le i\le m, 0\le j\le p, i+j=k}.
\end{equation}
\end{prop}

\begin{proof} Since ${\bf 1}\shap [\vec s]=[\vec s]=[\vec s]\shap {\bf 1}$, by Remark \ref{rem:wei>pwei}, we have
$$w_k([\vec s])\ge \tilde w_k([\vec s])=\tilde w_0({\bf 1})+\tilde w_k([\vec s])
$$
for $k\ge 1$. So we have the conclusion for $m=0$ or $p=0$. In general,
we prove the proposition by induction on $m+p\geq 0$, with the cases of $m+p =0,1$ checked directly.

 Now assume the conclusion holds for $2\le m+p \le \ell$. Then for $m+p=\ell+1$, we already know that, when $m=0$ or $p=0$, the conclusion is true, so we only need to consider the case $m\ge 1$ and $p\ge 1$. Let $(s_1, \vec s) \in \Z ^m$, $(t_1, \vec t)\in \Z ^p$.  We prove the conclusion for $[s_1, \vec s]$ and $[t_1, \vec t]$, assuming the conclusion for $m+p=\ell$. We divide the proof in four cases, with Case 1 serving as the initial step of the inductive  proofs of later cases.

\noindent
{\bf Case 1:} If $s_1=0$, let
$$[\vec s] \shap [t_1, \vec t]=\sum_{[\vec{v}]} a_{[\vec v]}[\vec v],
$$
with $a_{[\vec v]}\not =0$. Then
$$[0, \vec s]\shap [t_1,\vec t]=[0, [\vec s] \shap [t_1, \vec t]]=\sum_{[\vec{v}]} a_{[\vec v]}[0,\vec v].
$$
By the induction hypothesis (on $m+p-1$), for $k\ge 1$,
$$w_k([\vec v])\ge {\rm min}\{\tilde w_i([\vec s])+\tilde w_j([t_1, \vec t])\}_{ 0\le i\le m-1, 0\le j\le p, i+j=k}.
$$
So for $k\ge 2$, we obtain
\begin{align*}
w_k([0, \vec v])=&w_{k-1}([\vec v])\\
\ge& {\rm min}\{\tilde w_i([\vec s])+\tilde w_j([t_1, \vec t]) \}_{0\le i\le m-1, 0\le j\le p, i+j=k-1}\\
=&{\rm min}\{\tilde w_{i+1}([0,\vec s])+\tilde w_j([t_1, \vec t])\}_{0\le i\le m-1, 0\le j\le p, i+j=k-1}\quad (\text{by~Lemma~\mref{eq:kandk-1}})\\
=&{\rm min}\{\tilde w_{i}([0,\vec s])+\tilde w_j([t_1, \vec t])\}_{1\le i\le m, 0\le j\le p, i+j=k}\\
\ge& {\rm min}\{\tilde w_{i}([0,\vec s])+\tilde w_j([t_1, \vec t])\}_{0\le i\le m, 0\le j\le p, i+j=k}.
\end{align*}

Also, for $k=1$, we have
\begin{align*}
w_1 ([0, \vec v])&=0=w_1([0, \vec s])\ge \tilde w_1([0, \vec s]) \ge \tilde w_1([0, \vec s])+\tilde w_0([t_1, \vec t]) \\
&\ge \rmmin \{\tilde w_1([0, \vec s])+\tilde w_0([t_1, \vec t]), \tilde w_0([0, \vec s])+\tilde w_1([t_1, \vec t])\}.
\end{align*}
Thus the inductive step holds in this case.

\noindent
{\bf Case 2:} If $s_1\ge 0$ and $ t_1\ge 0$, then let
$$Z_n:=\{(s_1, t_1)\in \Z_{\ge 0}\times \Z_{\ge 0}\ |\ s_1+t_1=n\},
$$
and let $T$ be the subset of $\Z_{\ge 0}\times \Z_{\ge 0}$ such that every term appearing in $[s_1,\vec s]\shap [t_1,\vec t]$ with nonzero coefficient satisfies the conditions in (\mref{eq:convergentmin}). By the case of $s_1=0$, we conclude $Z_0\subset T$.

Now assuming $ Z_n\subset T$,  we consider $(s_1, t_1)\in Z_{n+1}$. We deal with this in three subcases.

\noindent
{\bf Subcase 2.1} When $s_1=0, t_1>0:$ This is already covered in Case 1.

\noindent
{\bf Subcase 2.2} When  $s_1>0, t_1=0:$ By definition,
$$[s_1, \vec s]\shap [0, \vec t]=[0, [s_1, \vec s]\shap [\vec t]].$$
So if
$$[s_1, \vec s]\shap [\vec t]=\sum a_{[\vec v]}[\vec v],$$
then
$$[s_1, \vec s]\shap [0, \vec t]=\sum a_{[\vec v]}[0,\vec v].$$
Thus
\begin{align*}
w_1([0, \vec v])&=0= w_1([0, \vec t])\ge \tilde w_1([0, \vec t])\ge \tilde w_0([s_1, \vec s])+\tilde w_1([0, \vec t])\\
&\ge \rmmin \{\tilde w_1([s_1, \vec s])+\tilde w_0([0, \vec t]), \tilde w_0([s_1, \vec s])+\tilde w_1([0, \vec t])\}.
\end{align*}
For $k\ge 2$, by the induction hypothesis that $Z_n\subset T$, we have
 \begin{align*}
   w_k([0, \vec v])&=w_{k-1}([\vec v])\\
   &\ge {\rm min}\{\tilde w_i([s_1,\vec s])+\tilde w_j([\vec t])\}_{0\le i\le m, 0\le j\le p-1, i+j=k-1}\\
   &\ge {\rm min}\{\tilde w_i([s_1,\vec s])+\tilde w_j([0,\vec t])\}_{0\le i\le m, 0\le j\le p, i+j=k},
\end{align*}
which is the conclusion in this subcase.

\noindent
{\bf Subcase 2.3} When $s_1>0, t_1>0:$ Definition ~\mref{de:extsha} gives
$$[s_1,\vec s]\shap [t_1, \vec t]=I([s_1-1,\vec s]\shap [t_1, \vec t]+[s_1, \vec s]\shap [t_1-1, \vec t]).
$$
Write
$$ [s_1-1, \vec s]\shap [t_1, \vec t]=\sum a_{[\vec v]}[\vec v].$$
 The induction hypothesis that $Z_n\subset T$ applies to this sum. So for $k\ge 1$, we have
\begin{align*}
w_k([\vec v]) &\ge  \rmmin \big \{\tilde w_k([s_1-1, \vec s])+\tilde w_0 ([t_1, \vec t]), \tilde w_0([s_1-1, \vec s])+\tilde w_k ([t_1, \vec t]),\\
&\qquad \qquad \qquad \tilde w_i([s_1, \vec s])-1+\tilde w_j ([t_1, \vec t]) \big \}_{1\le i\le m, 1\le j\le p, i+j=k}\\
&=  \rmmin \big \{\tilde w_k([s_1-1, \vec s]), \tilde w_k ([t_1, \vec t]),\\
&\qquad \qquad \qquad \tilde w_i([s_1, \vec s])-1+\tilde w_j ([t_1, \vec t])\big \}_{1\le i\le m, 1\le j\le p, i+j=k}.
\end{align*}
Then
\begin{align*}
w_k(I[\vec v]) \ge&  \rmmin \big \{\tilde w_k([s_1-1, \vec s])+1, \tilde w_k([t_1, \vec t])+1,\\
&\qquad \qquad \qquad \tilde w_i([s_1, \vec s])+\tilde w_j ([t_1, \vec t]) \big \}_{ 1\le i\le m, 1\le j\le p, i+j=k} \\
\ge& \rmmin \big \{\tilde w_k([s_1, \vec s])+\tilde w_0([t_1, \vec t]), \tilde w_0([s_1, \vec s])+\tilde w_k([t_1, \vec t]),\\
&\qquad \qquad \qquad \tilde w_i([s_1,\vec s])+\tilde w_j ([t_1, \vec t])\big \}_{1\le i\le m. 1\le j\le p, i+j=k},
\end{align*}
 where the first inequality follows from $w_k(I[\vec v])=w_k([\vec v])+1$, and the second inequality follows from $\tilde w_k([s_1-1, \vec s])+1=\tilde w_k([s_1, \vec s])\ge \tilde w_k([s_1, \vec s])+\tilde w_0([t_1, \vec t])$ and $\tilde w_k([t_1, \vec t])+1\ge \tilde w_k([t_1, \vec t])+\tilde w_0([s_1, \vec s])$. Thus every term from this part satisfies the condition. Similarly we have the conclusion for the terms from $I([s_1, \vec s]\shap [t_1-1, \vec t])$.

Therefore, $Z_{n+1}\subset T$. Thus $T=\Z_{\ge 0}\times \Z_{\ge 0}$, which means that the conclusion holds in Case 2.

\noindent
{\bf Case 3:} If $s_1 \ge 0$ and $t_1\le 0$, then denote
$$Z_n:=\{(s_1, t_1)\in \Z_{\ge 0}\times \Z_{\le 0} \ |\ t_1=-n\}.
$$
Let $T$ be the subset of $\Z_{\ge 0}\times \Z_{\le 0}$ such that every term $[\vec v]$ of $[s_1,\vec s]\shap [t_1,\vec t]$ satisfies the conditions in (\mref{eq:convergentmin}). We apply induction on $|t_1|$.
The initial step $t_1=0$ follows from Case 2, yielding $Z_0\subset T$.
If $t_1<0$, inductively assume  $Z_n\subset T$. Then for $(s_1, t_1)\in Z_{n+1}$, we have $(0, t_1)\in T$ by Case 1 of the proof. So we can assume $s_1>0$ and $t_1<0$. Then
$$[s_1,\vec s]\shap [t_1, \vec t]=J([s_1, \vec s]\shap [t_1+1,\vec t])-[s_1-1,\vec s]\shap [t_1+1, \vec t],
$$
here $(s_1, t_1+1), (s_1-1, t_1+1)\in Z_n$. So every term $[\vec v]$  of  $ [s_1, \vec s]\shap [t_1+1, \vec t]$  satisfies
$$w_k([\vec v]) \ge  \rmmin\{\tilde w_i([s_1, \vec s])+\tilde w_j([t_1+1, \vec t])\}_{i+j=k}.
$$
Then
 \begin{align*}
&w_k(J[\vec v]) \ge  \rmmin\{\tilde w_i([s_1, \vec s])+\tilde w_j([t_1+1, \vec t])\}_{0\le i\le m, 0\le j\le p, i+j=k}-1\\
=&\rmmin\{\tilde w_i([s_1, \vec s])+\tilde w_j([t_1, \vec t])\}_{0\le i\le m, 0\le j\le p, i+j=k}.
\end{align*}
At the same time, the induction hypothesis yields $(s_1-1, t_1+1)\in Z_n\subset T$.
Note that a term $[\vec u]$ of  $[s_1-1, \vec s]\shap [t_1+1, \vec t]$ is also a term of  $[s_1,\vec s]\shap [t_1, \vec t]$.  So it satisfies
 \begin{align*}
&w_k([\vec u]) \ge  {\rmmin}\{\tilde w_i([s_1-1, \vec s])+\tilde w_j([t_1+1, \vec t])\}_{0\le i\le m, 0\le j\le p, i+j=k}\\
=&\rmmin\{\tilde w_k([s_1-1, \vec s])+\tilde w_0([t_1+1, \vec t]), \tilde w_0([s_1-1, \vec s])+\tilde w_k([t_1+1, \vec t]), \\
&\qquad \qquad \qquad \tilde w_i([s_1-1, \vec s])+\tilde w_j([t_1+1, \vec t])\}_{1\le i\le m, 1\le j\le p, i+j=k}\\
=&\rmmin\{\tilde w_k([s_1-1, \vec s])+t_1+1, \tilde w_k([t_1+1, \vec t]),\\
 &\qquad \qquad \qquad \tilde w_i([s_1, \vec s])+\tilde w_j([t_1, \vec t])\}_{ 1\le i\le m, 1\le j\le p, i+j=k}\\
=&\rmmin\{\tilde w_k([s_1, \vec s])+\tilde w_0([t_1, \vec t]), \tilde w_0([s_1, \vec s])+\tilde w_k([t_1+1, \vec t]),\\
&\qquad \qquad \qquad \tilde w_i([s_1, \vec s])+\tilde w_j([t_1, \vec t])\}_{ 1\le i\le m, 1\le j\le p, i+j=k}\\
\ge&\rmmin\{\tilde w_k([s_1, \vec s])+\tilde w_0([t_1, \vec t]), \tilde w_0([s_1, \vec s])+\tilde w_k([t_1, \vec t]),\\
&\qquad \qquad \qquad \tilde w_i([s_1, \vec s])+\tilde w_j([t_1, \vec t])\}_{ 1\le i\le m, 1\le j\le p, i+j=k}\\
=&\rmmin\{\tilde w_i([s_1, \vec s])+\tilde w_j([t_1, \vec t])\}_{0\le i\le m, 0\le j\le p, i+j=k}.
\end{align*}
Hence  $Z_{n+1}\subset T$. So $T=Z_{\ge 0}\times \Z_{\le 0}$.

{\bf Case 4:} If $s_1\le 0$, then denote
$$Z_n:=\{(s_1, t_1)\in \Z_{\le 0}\times \Z \ |\ s_1=-n\}.
$$
Let $T$ be the subset of $\Z_{\le 0}\times \Z$ such that every term of $[s_1,\vec s]\shap [t_1,\vec t]$ satisfies the condition in  Eq. (\mref{eq:convergentmin}). We apply induction on $|s_1|\geq 0$.

The initial step $s_1=0$ is already considered in Case 1, so we have $Z_0\subset T$.
For the induction step of $s_1<0$, assume $Z_n\subset T$ for $n\geq 0$ and consider $(s_1, t_1)\in Z_{n+1}$.
First we have
$$[s_1,\vec s]\shap [t_1, \vec t]=J([s_1+1, \vec s]\shap [t_1,\vec t])-[s_1+1,\vec s]\shap [t_1-1, \vec t].
$$
By the induction  hypothesis on $|s_1|$,
every term $[\vec v]$  of  $[s_1+1, \vec s]\shap [t_1, \vec t]=\sum a_{[\vec v]}[\vec v]$  satisfies
$$w_k([\vec v])\ge \rmmin\{\tilde w_i([s_1+1, \vec s])+\tilde w_j([t_1, \vec t])\}_{0\le i\le m, 0\le j\le p, i+j=k}
$$
for each $1\le k\le \ell+1$. So applying $J$ gives
\begin{align*}
w_k(J[\vec v]) &\ge  \rmmin\{\tilde w_i([s_1+1, \vec s])+\tilde w_j([t_1, \vec t])\}_{0\le i\le m, 0\le j\le p, i+j=k}-1\\
&=\rmmin\{s_1+\tilde w_k([t_1, \vec t]), \tilde w_i([s_1, \vec s])+\tilde w_j([t_1, \vec t])\}_{ 1\le i\le m, 0\le j\le p, i+j=k}\\
&=\rmmin\{\tilde w_i([s_1, \vec s])+\tilde w_j([t_1, \vec t])\}_{0\le i\le m, 0\le j\le p, i+j=k}.
\end{align*}
So these terms satisfy Eq.~\meqref{eq:convergentmin}.

Also by the induction hypothesis, every term $[\vec u]$  of  $[s_1+1, \vec s]\shap [t_1-1, \vec t]=\sum a_{[\vec u]}[\vec u]$ satisfies
$$w_k([\vec u]) \ge  \rmmin\{\tilde w_i([s_1+1, \vec s])+\tilde w_j([t_1-1, \vec t])\}_{0\le i\le m, 0\le j\le p, i+j=k}
$$
for each $1\le k\le \ell+1$. To show that they also satisfy Eq.~\meqref{eq:convergentmin}, we consider two subcases.

\begin{enumerate}
\item ({\bf Subcase 3.1}) When $t_1\ge 1$, we have
\begin{align*}
w_k([\vec u]) \ge&  \rmmin\{\tilde w_k([s_1+1, \vec s]), s_1+1+\tilde w_k([t_1-1, \vec t]),\\
&\qquad \qquad \qquad \tilde w_i([s_1+1, \vec s])+\tilde w_j([t_1-1, \vec t])\}_{1\le i\le m, 1\le j\le p, i+j=k}\\
\ge& \rmmin\{\tilde w_k([s_1, \vec s])+\tilde w_0([t_1, \vec t]), \tilde w_0([s_1, \vec s])+\tilde w_k([t_1, \vec t]),\\
&\qquad \qquad \qquad \tilde w_i([s_1, \vec s])+\tilde w_j([t_1, \vec t])\}_{1\le i\le m, 1\le j\le p, i+j=k}\\
=& \rmmin\{\tilde w_i([s_1, \vec s])+\tilde w_j([t_1, \vec t])\}_{0\le i\le m, 0\le j\le p, i+j=k}.
\end{align*}
This is what we need.
\item ({\bf Subcase 3.2}) When $t_1<1$, we have
\begin{align*}
w_k([\vec u]) \ge&  \rmmin\{\tilde w_k([s_1+1, \vec s])+t_1-1, s_1+1+\tilde w_k([t_1-1, \vec t]),\\
&\qquad \qquad \qquad \tilde w_i([s_1+1, \vec s])+\tilde w_j([t_1-1, \vec t])\}_{1\le i\le m, 1\le j\le p, i+j=k}\\
=& \rmmin\{\tilde w_k([s_1, \vec s])+\tilde w_0([t_1, \vec t]), \tilde w_0([s_1, \vec s])+\tilde w_k([t_1, \vec t]),\\
&\qquad \qquad \qquad \tilde w_i([s_1, \vec s])+\tilde w_j([t_1, \vec t])\}_{1\le i\le m, 1\le j\le p, i+j=k}\\
=& \rmmin\{\tilde w_i([s_1, \vec s])+\tilde w_j([t_1, \vec t])\}_{0\le i\le m, 0\le j\le p, i+j=k}.
\end{align*}
\end{enumerate}
In summary, we have prove $Z_{n+1}\subset T$. Hence the induction gives $T=\Z_{\le 0}\times \Z$.
This completes the proof.
\end{proof}

We finally obtain
\begin{theorem}
\mlabel{thm:CAlg}
$(\calh_{\Z}^0, \shap)$ is a subalgebra of $(\calha, \shap)$.
\end{theorem}

\begin{proof} Obviously,
$${\bf 1}\shap a=a\shap {\bf 1}=a \in \calh_{\Z}^0, \quad a\in \calh_{\Z}^0.
$$

For $(s_1, \vec s)\in \Z ^m$ and $(t_1, \vec t)\in \Z ^p$, if $[s_1, \vec s], [t_1, \vec t] \in \calh_{\Z}^0$, then for $1\le i\le m$ and $1\le j\le p$, we have
$$w_i([s_1, \vec s])>i, \ w_j([t_1, \vec t])>j.
$$
By  Proposition  \mref{lem:termgemin},  every term  $[\vec v]$ of $[s_1, \vec s]\shap [t_1,\vec t]=\sum a_{[\vec v]}[\vec v]$ satisfies the conditions
$$
w_k([\vec v])\ge \rmmin\{\tilde w_i([s_1, \vec s])+\tilde w_j([t_1, \vec t])\}_{0\le i\le m, 0\le j\le p, i+j=k}, \quad 1\le k\le m+p.
$$
 With the notations in Eq.~\meqref{eq:tildepw}, we obtain
{\small \begin{align*}
&\tilde w_i([s_1, \vec s])+\tilde w_j([t_1, \vec t])\\
=&\left\{\begin{array}{ll}
w_i([s_1, \vec s])+w_j([t_1, \vec t])>i+j=k, &\text{if either } s_{i+1}>0\ \text{ or}\ i=m,
 \text{ and either } t_{j+1}>0 \ {\rm or}\ j=p,\\
w_{i+1}([s_1, \vec s])+w_j([t_1, \vec t])>k+1, &\text{if } s_{i+1}\le 0, \text{ and either } t_{j+1}>0\ {\rm or}\ j=p,\\
w_{i}([s_1, \vec s])+w_{j+1}([t_1, \vec t])>k+1, &\text{if either } s_{i+1}>0 \text{ or}\ i=m, \text{ and } t_{j+1}\le 0,\\
w_{i+1}([s_1, \vec s])+w_{j+1}([t_1, \vec t])>k+2, &\text{if }  s_{i+1}\le 0 \text{ and } t_{j+1}\le 0.
\end{array}\right.
\end{align*}}
So the minimal is always greater than $k$. Thus the multiple zeta series
$$\zeta (\vec v)=\sum_{x_i\in \Z _{>0}}\frac{1}{(x_1+\cdots+x_k)^{s_1}\cdots x_k^{s_k}}$$
converges, and hence
$$[s_1,\vec s]\shap [t_1,\vec t]\in \calh_{\Z}^0.$$
Therefore, $(\calh_{\Z}^0, \shap)$ is a subalgebra of $(\calha, \shap)$.
\end{proof}

\subsection {Algebra homomorphism}
By Theorem  \mref{thm:CAlg}, $(\calh_{\Z}^{0}, \shap)$ is an algebra. The assignment
$$\zeta([s_1,\cdots, s_k])=\sum_{x_i\in \Z _{>0}}\frac{1}{(x_1+\cdots+x_k)^{s_1}\cdots x_k^{s_k}}, \quad [s_1,\cdots, s_k]\in \calh_{\Z}^{0},
$$
defines a linear map from $\calh_{\Z}^{0}$ to $\R$. For the variable set $\{x_i\}_{i\in \Z _{\ge 1}}$, define
$$S_c^{\mathrm{Ch}}:=\Big\{{\bf1}, \wvec{s_1, \cdots, s_m}{i_1, \cdots, i_m}\in \sch\ \Big |\ s_1+\cdots+s_j>j, j=1,\cdots, m\Big\},
$$
$$F_c^{\mathrm{Ch}}:=\Big\{{\bf1}, f \wvec{s_1, \cdots, s_m}{x_{i_1},\cdots, x_{i_m}}\in\fch\ \Big|\ s_1+\cdots+s_j>j, j=1,\cdots, m\Big\}.
$$

\begin{prop}\label{prop:Chenconve}
$(\Q S_c^{\mathrm{Ch}}, \shap)$ is a locality subalgebra of $(\Q \sch, \shap)$, and $(\Q F_c^{\mathrm{Ch}}, \cdot)$ is a subalgebra of $(\Q \fch, \cdot)$.
\end{prop}

\begin{proof}
An argument similar to the proof of Theorem \mref{thm:CAlg} shows that $(\Q S_c^{\mathrm{Ch}}, \shap)$ is a locality subalgebra. It is clear that $\Q F_c^{\mathrm{Ch}}=F(\Q S_c^{\mathrm{Ch}})$ for the map $F$ defined in Eq.~\meqref{eq:fmap}. Then by Proposition \mref{prop:zhomo}, $\Q F_c^{\mathrm{Ch}}$ is a subalgebra of $\Q\fch$.
\end{proof}

Now we define a linear map
\begin{align*}
\eta:  \Q F_c^{\mathrm{Ch}}&\longrightarrow \R, \\
{\bf 1}&\longmapsto 1, \\
f\wvec{s_1, \cdots, s_m}{x_{i_1}, \cdots, x_{i_m}}& \longmapsto \sum_{x_{i_j>0}} \frac{1}{(x_{i_1}+\cdots+x_{i_k})^{s_1}\cdots x_{i_k}^{s_k}}.
\end{align*}
Then
\begin{lemma}
\mlabel{lemma:eta}
For $a, b\in \Q F_c^{\mathrm{Ch}}$, if $a\top b$, then $\eta$ is a locality algebra homomorphism$:$
$$\eta(a b)=\eta(a)\eta(b).
$$
\end{lemma}

\begin{proof}
We only need to prove the identity on basis elements. If $a={\bf 1}$ or $b={\bf 1}$, the conclusion holds obviously. Otherwise, let
$$a=f\wvec{s_1, \cdots, s_m}{x_{i_1}, \cdots, x_{i_m}}, b=f\wvec{t_1, \cdots, t_n}{x_{j_1}, \cdots, x_{j_n}},
$$
where $\{i_1, \cdots, i_m\}\cap \{j_1, \cdots, j_n\}=\emptyset$. Since $f\wvec{s_1, \cdots, s_m}{x_{i_1}, \cdots, x_{i_m}}$, $f\wvec{t_1, \cdots, t_n}{x_{j_1}, \cdots, x_{j_n}}\in F_c^{\mathrm{Ch}}$, by Proposition \ref{prop:Chenconve},
$$f\wvec{s_1, \cdots, s_m}{x_{i_1}, \cdots, x_{i_m}}\cdot f\wvec{t_1, \cdots, t_n}{y_{j_1}, \cdots, y_{j_n}}\in \Q F_c^{\mathrm{Ch}}.
$$

Let
\begin{equation*}
\wvec{s_1, \cdots, s_m}{{i_1}, \cdots,{i_m}} \shap \wvec{t_1, \cdots, t_n}{{j_1}, \cdots, {j_n}}=\sum A_{\wvec{u_1, \cdots, u_{m+n}}{\ell_1, \cdots, \ell_{m+n}}}\wvec{u_1, \cdots, u_{m+n}}{\ell_1, \cdots, \ell_{m+n}}
\end{equation*}
with
$A_{\wvec{u_1, \cdots, u_{m+n}}{\ell_1, \cdots, \ell_{m+n}}}\in \Z$ and $\ell_1, \cdots, \ell_{m+n}$ a shuffle of $i_1, \cdots , i_m, j_1, \cdots, j_n$ by Proposition~\mref{p:tworowprod}.
Then apply $F$, by Proposition \ref{prop:zhomo}, we have
\begin{equation}\label{eq:Symbst}
f\wvec{s_1, \cdots, s_m}{x_{i_1}, \cdots, x_{i_m}}\cdot f\wvec{t_1, \cdots, t_n}{x_{j_1}, \cdots, x_{j_n}}=\sum A_{\wvec{u_1, \cdots, u_{m+n}}{\ell_1, \cdots, \ell_{m+n}}}f\wvec{u_1, \cdots, u_{m+n}}{x_{\ell_1}, \cdots, x_{\ell_{m+n}}}.
\end{equation}
Thus
\begin{equation*}
\begin{split}
&\eta (f\wvec{s_1, \cdots, s_m}{x_{i_1}, \cdots, x_{i_m}}\cdot f\wvec{t_1, \cdots, t_n}{x_{j_1}, \cdots, x_{j_n}})\overset{(\ref{eq:Symbst})}{=}\eta \bigg(\sum A_{\wvec{u_1, \cdots, u_{m+n}}{\ell_1, \cdots, \ell_{m+n}}}f\wvec{u_1, \cdots, u_{m+n}}{x_{\ell_1}, \cdots, x_{\ell_{m+n}}}\bigg)\\
=&\sum A_{\wvec{u_1, \cdots, u_{m+n}}{\ell_1, \cdots, \ell_{m+n}}} \eta (f\wvec{u_1, \cdots, u_{m+n}}{x_{\ell_1}, \cdots, x_{\ell_{m+n}}})\\
=&\sum A_{\wvec{u_1, \cdots, u_{m+n}}{\ell_1, \cdots, \ell_{m+n}}} \sum _{x_{i_1}, \cdots x_{i_m}, x_{j_1}, \cdots , x_{j_n}} f\wvec{u_1, \cdots, u_{m+n}}{x_{\ell_1}, \cdots, x_{\ell_{m+n}}}\\
=&\sum _{x_{i_1}, \cdots x_{i_m}, x_{j_1}, \cdots , x_{j_n}} \sum A_{\wvec{u_1, \cdots, u_{m+n}}{\ell_1, \cdots, \ell_{m+n}}}  f\wvec{u_1, \cdots, u_{m+n}}{x_{\ell_1}, \cdots, x_{\ell_{m+n}}}\\
\overset{(\ref{eq:Symbst})}{=}&\sum _{x_{i_1}, \cdots x_{i_m}, x_{j_1}, \cdots , x_{j_n}} f\wvec{s_1, \cdots, s_m}{x_{i_1}, \cdots, x_{i_m}}\cdot f\wvec{t_1, \cdots, t_n}{x_{j_1}, \cdots, x_{j_n}}\\
=&\eta(f\wvec{s_1, \cdots, s_m}{x_{i_1}, \cdots, x_{i_m}})\eta(f\wvec{t_1, \cdots, t_n}{x_{j_1}, \cdots, x_{j_n}}).
\end{split}
\end{equation*}
Hence $\eta$ is a locality algebra homomorphism.
\end{proof}

 So we have the following diagram which is commutative by directly checking
$\zeta^\shap \Phi=\eta F$.
  $$\xymatrix{
 \Q S_c^{\mathrm{Ch}} \ar[d]_{F} \ar[rr]^{\Phi}
                && \calh_{\Z}^{0} \ar[d]^{\zeta^\shap}  \\
  \Q F_c^{\mathrm{Ch}} \ar[rr]_{\eta}
                && \R  .         }$$

\begin{theorem}
\mlabel {thm:zhomo} The linear map $\zeta^\shap: \calh _{\Z }^0\to \R$ is an algebra homomorphism.
\end{theorem}

\begin {proof} For $[s_1, \cdots, s_m], [t_1, \cdots , t_p]\in \calh _{\Z }^0$, we have
\begin{align*}
&\zeta^\shap ([s_1, \cdots, s_m])\zeta^\shap ( [t_1, \cdots , t_p])\\
=&\big(\zeta^\shap \Phi\wvec{s_1, \cdots, s_m}{1, \cdots , m}\big)\big(\zeta^\shap \Phi\wvec{t_1, \cdots, t_p}{{m+1}, \cdots , {m+p}}\big) \quad \text{(by the definition of } \Phi)\\
=&(\eta F\wvec {s_1, \cdots, s_m}{1, \cdots , m})(\eta F\wvec {t_1, \cdots, t_p}{m+1, \cdots , m+p})\quad \text{(by}~\wvec {s_1, \cdots, s_m}{1, \cdots , m}\top \wvec {t_1, \cdots, t_p}{m+1, \cdots , m+p}, \zeta^\shap \Phi=\eta F)\\
=&\eta\big(F\wvec {s_1, \cdots, s_m}{1, \cdots , m} F\wvec {t_1, \cdots, t_p}{m+1, \cdots , m+p}\big)\quad (\text{by Lemma}~ \mref{lemma:eta})\\
=&\eta F(\wvec {s_1, \cdots, s_m}{1, \cdots , m}\shap \wvec {t_1, \cdots, t_p}{m+1, \cdots , m+p})\quad (\text{by Proposition}~ \mref{prop:zhomo})\\
=&\zeta^\shap\Phi(\wvec {s_1, \cdots, s_m}{1, \cdots , m}\shap \wvec {t_1, \cdots, t_p}{m+1, \cdots , m+p})\quad \text{(by}~\eta F=\zeta^\shap \Phi)\\
=&\zeta^\shap(\Phi\wvec {s_1, \cdots, s_m}{1, \cdots , m}\shap \Phi\wvec {t_1, \cdots, t_p}{m+1, \cdots , m+p})\quad (\text{by Proposition}~ \mref{prop:phihomo})\\
=&\zeta^\shap ([s_1, \cdots, s_m]\shap [t_1, \cdots, t_p]). \quad \text{(by the definition of } \Phi)\qedhere
\end{align*}
\end{proof}

\noindent
{\bf Acknowledgments.} This research is supported by
NNSFC (12471062).

\noindent
{\bf Declaration of interests. } The authors have no conflict of interest to declare that are relevant to this article.

\noindent
{\bf Data availability. } Data sharing is not applicable to this article as no data were created or analyzed.

\end{document}